\documentclass[reqno, 11pt]{amsart}
\usepackage[utf8]{inputenc}
\usepackage{amssymb,amsxtra,latexsym}
\usepackage{ragged2e}\justifying\let\raggedright\justifying
\usepackage{mathrsfs}
\usepackage{graphicx}
\usepackage{hyperref}
\usepackage{textcomp}
\usepackage[scr=rsfs]{mathalfa}
\usepackage{epstopdf}
\usepackage{tikz}
\usepackage{pgfplots}
\usepackage{tikz-cd}
\usepackage[text={155mm,235mm},centering]{geometry}
\usepackage[mathscr]{eucal}
\usepackage[all]{xy}
\usepackage{mathrsfs}
\usepackage{xypic}
\usepackage{amsfonts}
\usepackage{amsmath}
\usepackage{amsthm,bm}
\usepackage{amssymb,tikz-cd}
\usepackage{latexsym}
\usepackage{tabularx}
\usepackage{graphicx}
\usepackage{pict2e}
\usepackage{tikz}
\usepackage{textcomp}
\usepackage[scr=rsfs]{mathalfa}
\usepackage[active]{srcltx}
\newtheorem{thm}{Theorem}[section]

\newtheorem{lem}[thm]{Lemma}
\newtheorem{prop}[thm]{Proposition}

\newtheorem*{prob}{Problem}

\theoremstyle{definition}
\newtheorem{defn}[thm]{Definition}
\theoremstyle{property}
\newtheorem{prope}[thm]{Property}
\theoremstyle{remark}
\newtheorem{rem}[thm]{Remark}

\newtheorem{ex}[thm]{Example}
\numberwithin{equation}{section}
\definecolor{ceruleanblue}{rgb}{0.16, 0.32, 0.75}
\hypersetup{colorlinks=true,allcolors=ceruleanblue}
\allowdisplaybreaks
\numberwithin{equation}{section}
\DeclareMathOperator{\C}{\mathcal{C}}

\DeclareMathOperator{\CO}{\mathcal{O}}

\DeclareMathOperator{\Hom}{\mathrm{Hom}}

\DeclareMathOperator{\DR}{\mathrm{DR}}
\DeclareMathOperator{\MC}{\mathcal{S}}
\makeatletter
\@namedef{subjclassname@2020}{%
  \textup{2020} Mathematics Subject Classification}
\makeatother

\makeatletter
\def\@tocline#1#2#3#4#5#6#7{\relax
  \ifnum #1>\c@tocdepth 
  \else
    \par \addpenalty\@secpenalty\addvspace{#2}%
    \begingroup \hyphenpenalty\@M
    \@ifempty{#4}{%
      \@tempdima\csname r@tocindent\number#1\endcsname\relax
    }{%
      \@tempdima#4\relax
    }%
    \parindent\z@ \leftskip#3\relax \advance\leftskip\@tempdima\relax
    \rightskip\@pnumwidth plus4em \parfillskip-\@pnumwidth
    #5\leavevmode\hskip-\@tempdima
      \ifcase #1
       \or\or \hskip 1em \or \hskip 2em \else \hskip 3em \fi%
      #6\nobreak\relax
    \dotfill\hbox to\@pnumwidth{\@tocpagenum{#7}}\par
    \nobreak
    \endgroup
  \fi}
\makeatother
\begin{document}

\title[Symplectic stratified pseudomanifolds]{Symplectic structures on stratified pseudomanifolds}

\author{Xiangdong Yang}
\address{Department of Mathematics, Lanzhou University, Lanzhou 730000, China}
\email{yangxd@lzu.edu.cn}


\date{\today}


\begin{abstract}
  The purpose of this paper is to investigate the definition of symplectic structure on a smooth stratified pseudomanifold in the framework of local $\C^{\infty}$-ringed space theory.
  We introduce a sheaf-theoretic definition of symplectic form and cohomologically symplectic structure on smooth stratified pseudomanifolds.
  In particular, we give an indirect definition of symplectic form on the quotient space of a smooth $G$-stratified pseudomanifold.
  Based on the structure theorem of singular symplectic quotients by Sjamaar--Lerman, we show that the singular reduced space $M_{0}=\mu^{-1}(0)/G$ of a symplectic Hamiltonian $G$-manifold $(M,\omega,G,\mu)$ admits a natural (indirect) symplectic form and a unique cohomologically symplectic structure.
\end{abstract}

\maketitle

\tableofcontents
\section{Introduction}
\subsection{Background}


To study the geometry of a singular space, for example an algebraic variety, a natural idea is to divide it into smooth manifolds.
As a geometric realization of this idea, Whitney \cite{Wh47} introduced the concept of a \lq\lq \emph{complifold}".
In the literature, it was Thom \cite{Th62} who first introduced the notion of \emph{stratification}.
In fact, Whitney's notion of complifold, or complex of manifolds, can be thought of as the first attempt at an abstract theory of stratifications.
Whitney was concerned with the triangulability of algebraic varieties, whereas Thom was concerned with the differentiable stability of smooth mappings.
Continuing the work by Thom, Mather completely worked out the differentiable stability problem in the stratification theory, and wrote a series of lecture notes which have become the standard reference of the stratification theory, see \cite{Ma12} and the references therein.
For a geometric study of stratified spaces, Goresky--MacPherson \cite{GM80,GM88} developed the intersection homology theory and stratified Morse theory to extend classical results such as the Poincar\'{e} duality and K\"{u}nneth formula to singular algebraic varieties.
In the study of stratified spaces, one of the most fundamental problems is that there does not exist a very good bundle theory.

In symplectic geometry, the stratification structure appeares naturally in the process of symplectic reduction.
Let $(M,\omega)$ be a connected symplectic manifold on which a compact Lie group $G$ acts in a Hamiltonian fashion.
Assume that the moment map $\mu:M\rightarrow\mathfrak{g}^{\ast}$ is proper.
Set $Z=\mu^{-1}(0)$ and $M_{0}=Z/G$.
With the induced sub-topology on $Z$ and quotient topologies on $M/G$ and $M_{0}$ respectively, we get a commutative square of continuous maps
\begin{equation*}
\vcenter{
\xymatrix@C=1.3cm{
  Z \,\ar[d]_{\pi} \ar[r]^{\imath} _{\mathrm{inclusion}}& M \ar[d]^{\Pi} \\
  M_{0}\, \ar[r]^{} & M/G.}
  }
\end{equation*}
The Marsden--Weinstein reduction procedure says that
if zero is a regular value of $\mu$ and $G$ acts freely on $Z$, then the \emph{symplectic quotient} (also called the \emph{reduced space}) $M_{0}$ is a symplectic manifold.
In general, the regular condition guarantees that the $G$-action on $Z$ is locally free and therefore $M_{0}$ becomes a symplectic orbifold.
If even the regular condition is removed the level set $Z$ has quadratic singularities and the reduced space $M_{0}$ can be a highly singular space.
Originating from the fundamental paper \cite{MW74}, the study of the structure of $M_{0}$ when it acquires singularities has been one of central problems in symplectic geometry.
In particular, from an algebro-geometric viewpoint, this problem closely relates to some topics
in algebraic geometry (cf. \cite{Ki98}).

From the viewpoint of Poisson geometry, Arms--Cushman--Gotay \cite{ACG89} showed that there exists a canonical smooth structure $\C^{\infty}(M_{0})$ which is a subalgebra of the algebra of continuous functions $\C^{0}(M_{0})$ and equipped with a natural Poisson bracket $\{-,-\}_{0}$ from the original one on $\C^{\infty}(M)$.
Subsequently, in their paper \cite{SL91}, Sjamaar--Lerman proved an important structure theorem for singular reduced spaces and precisely clarified the nature
and proper definition of such spaces called stratified symplectic spaces.
This structure theorem  says that $M_{0}$ is a stratified space with symplectic manifolds as its strata;
moreover, all strata of $M_{0}$ fit together \emph{symplectically}.
To be more specific, the inclusion of each stratum $S\hookrightarrow M_{0}$ is a Poisson embedding with respect to the bracket $\{-,-\}_{0}$.
Such a stratified symplectic structure has been applied successfully to describing the geometry and topology of many moduli spaces (cf.  \cite{GHJW97,Hu95}).
In the singular case, it is worth noting that there exist various definitions of reduction with different emphasises (cf. \cite{AJG90}).

Note that the category of symplectic Hamiltonian $G$-manifolds is not closed under the process of symplectic reduction.
On the one hand, Arms--Cushman--Gotay pointed out that the Poisson bracket $\{-,-\}_{0}$ is non-degenerate and hence gives a \lq\lq \emph{symplectic structure}" on $M_{0}$.
On the other hand, the structure theorem \cite[Theorem 2.1]{SL91} shows that the original symplectic form $\omega$ on $M$ induces an ordinary symplectic form on each stratum of $M_{0}$.
So a natural problem is:
\begin{prob}\label{pro}
On the singular symplectic quotient $M_{0}$, what should a symplectic form be?
\end{prob}
As mentioned in \cite{BGL05}, it is an old question to define the symplectic structure on a singular space, for instance, a singular algebraic variety.
Since the reduced space $M_{0}$ is highly singular, there is no obvious differential-geometric method to define symplectic form on it.
Comparing with the symplectic structure, the Poisson structure is a purely algebraic structure which can be extended to singular spaces easily.
Working with the Poisson algebra structure, the symplectic structure on $M_{0}$ can be considered as a Poisson bracket on $\C^{\infty}(M_{0})$ which induces a usual symplectic structure on each stratum, see \cite[Definition 1.12]{SL91} and \cite[\S\,2.6]{Pf01}.
Recently, Mj--Sen \cite[Definition 4.1]{MS22} proposed an extrinsic definition of symplectic structure on a Whitney stratified space, which coincides with the stratified symplectic structure in \cite{SL91}.

In complex geometry, the problem of generalizing the definition of K\"{a}hler structure to the singular case has undergone a long history.
The first generalization in this direction can be traced back at least to Grauert \cite{Gr62} who introduced the K\"{a}hler metric on a complex analytic space.
Subsequently,  Moishezon \cite{Mo75} gave another concept of K\"{a}hler complex space, which was refined by Fujiki \cite{Fu78} and Varouchas  \cite{Va89}.
A stratified version of K\"{a}hler complex space was introduced by Heinzner--Huckleberry--Loose \cite{HHL94}; moreover, they studied the structure of symplectic quotient via the K\"{a}hlerian extension of a symplectic manifold.
Particularly, they showed that the stratified symplectic structure on the reduced space coincides with the restriction of the stratified K\"{a}hler structure on the quotient of its K\"{a}hlerian extension.
For a singular toric variety, Burns--Guillemin--Lerman \cite{BGL08} considered the K\"{a}hler structure on it and gave a K\"{a}hler potential formula which generalizes the formula in \cite{Gui94} for smooth toric varieties.
In \cite{HS20}, Heinzner--Stratmann proved that the symplectic reduction
of a K\"{a}hler manifold is a reduced normal complex space which admits a natural K\"{a}hler structure induced from the original K\"{a}hler metric.

Observe that the singularity of $M_{0}$ is a composite of two kinds of singularities: the intersection singularities of $Z=\mu^{-1}(0)$ and the quotient singularities of $M/G$.
In the framework of derived algebraic geometry, Calaque \cite[\S\,2.1.2]{Ca18} pointed out that the derived symplectic reduced space of a symplectic Hamiltonian $G$-manifold admits a natural \emph{0-shifted symplectic form} in the sense of Pantev--To\"{e}n--Vaqui\'{e}--Vezzosi \cite{PTVV13}.

Heuristically, in the stratified setting a \lq\lq \emph{symplectic form}" should canonically induce an ordinary symplectic form on each stratum.
Recall that a differential 2-form $\omega$ on an even-dimensional smooth manifold $M$ is non-degenerate, if and only if the induced morphism from the tangent bundle to the cotangent bundle
\begin{equation}\label{bd-iso}
\omega^{\natural}:TM\longrightarrow T^{\ast}M
\end{equation}
is an isomorphism.
If we view $M$ as a reduced local $\mathcal{C}^{\infty}$-ringed space with the structure sheaf of smooth functions $\C^{\infty}_{M}$,
then the sheaves of smooth sections for $TM$ and $T^{\ast}M$ coincide with the tangent sheaf $\mathcal{T}_{M}$ and the cotangent sheaf $\Omega^{1}_{M}$, respectively.
In the language of sheaf theory, the bundle map \eqref{bd-iso} is an isomorphism if and only if so is the associated \emph{sheaf morphism}
\begin{equation*}
\omega^{\flat}:\mathcal{T}_{M}\longrightarrow \Omega^{1}_{M}.
\end{equation*}
Before finding a meaningful definition of symplectic form on a stratified space, we need to understand what the smooth structure analogue is on it.
Observe that there is no good tangent bundle theory for a stratified space for the presence of singular strata, and the local structure of a stratified space changes from point to point.
Based on the works \cite{Hi69,Pf01}, we can equip a stratified space with a smooth structure and makes it into a reduced local $\C^{\infty}$-ringed space.
This indicates a \emph{sheaf-theoretic} definition of symplectic form on a stratified spaces from the viewpoint of $\C^{\infty}$-algebraic geometry (cf. \cite{Jo19}).

\subsection{Summary of the result}
In this paper, by taking a sheaf-theoretic approach, we introduce three different definitions of symplectic structures in the stratified setting: a notion of symplectic form on a smooth stratified pseudomanifold (Definition \ref{sym-spmd}), an indirect definition of symplectic form on the quotient space of a smooth stratified pseudomanifold by a smooth $G$-action (Definition \ref{ind-ss}), and the cohomologically symplectic structure (Definition \ref{c-sym}).
A benefit of the language of sheaves is that it allows us to carry out a unified treatment of smooth manifolds and stratified spaces in the category of local $\C^{\infty}$-ringed spaces.
From the viewpoint of $\C^{\infty}$-ringed space theory, we can endow the level set of a moment map (resp. the symplectic quotient) with a natural smooth structure which makes it into a local $\C^{\infty}$-ringed space.
The main result of this paper is the following

\begin{thm}\label{thm1}
Let $G$ be a compact connected Lie group and $(M,\omega)$ a Hamiltonian $G$-manifold with proper moment map
$\mu:M\rightarrow\mathfrak{g}^{\ast}$.
Then we have:
\begin{itemize}
  \item [(i)] the pullback $\underline{\imath}^{\ast}\omega$ gives rise to a symplectic structure on the quotient space $M_{0}=Z/G$ in the sense of Definition \ref{ind-ss},
      where $\underline{\imath}: (M, \C^{\infty}_{M})\hookrightarrow (Z, \C^{\infty}_{Z})$ is the inclusion;
  \item [(ii)] there exists a unique cohomologically symplectic structure $[\omega_{0}]\in\check{\mathrm{H}}^{2}(M_{0}, \underline{\mathbb{R}})$ such that $$\underline{\pi}^{\ast}([\omega_{0}])=\underline{\imath}^{\ast}([\omega]),$$
      where $\underline{\pi}: (Z, \C^{\infty}_{Z})\rightarrow (M_{0}, \C^{\infty}_{M_{0}})$ is the quotient map and $[\omega]$ is considered as the \v{C}ech cohomology class via the de Rham theorem $\check{\mathrm{H}}^{2}(M,\underline{\mathbb{R}})\cong \mathrm{H}^{2}_{\mathrm{DR}}(M)$;
  \item [(iii)] the Poisson bracket introduced by Arms--Cushman--Gotay induces an injective sheaf morphism
      $$
      \chi^{\ddag}:\Omega^{1}_{M_{0},\mathrm{tf}}\longrightarrow\mathcal{T}_{M_{0}},
      $$
      where $\Omega^{1}_{M_{0},\mathrm{tf}}$ is the torsion-free cotangent sheaf of $M_{0}$ and $\mathcal{T}_{M_{0}}$ is the tangent sheaf.
\end{itemize}
\end{thm}

This result is inspired by the structure theorem for singular reduced spaces \cite{SL91} and the treatment for singular symplectic quotient from complex geometry point of view via the Hamiltonian K\"{a}hlerian extension \cite{HHL94,HS20}.
Our motivation for introducing symplectic stratified pseudomanifold is to present a more algebro-geometric description of the symplectic structure on a singular symplectic reduced space.
Most of the important notions of symplectic manifolds can be extended to symplectic stratified pseudomanifolds.
Specifically, the point we want to make here is that the study of symplectic and complex stratified pseudomanifolds may be of independent interest.

\subsection{Structure of the paper}
In Section \ref{pre}, we present a brief review on some basic notations and the stratification structures. We devote Section \ref{smo} to the definition of $\C^{\infty}$-singular atlas on stratified spaces and give some basic properties of smooth stratified pseudomanifolds.
In Section \ref{sym-kah}, we introduce the notion of symplectic/K\"{a}hler structures on smooth stratified pseudomanifolds.
In Section  \ref{sig-kah}, as a special example of K\"{a}hler stratified pseudomanifold, we study the singular K\"{a}hler spaces.
In Section \ref{sym-quo}, we introduce an indirect definition of symplectic structure on the quotient space of a smooth $G$-stratified pseudomanifold, and we examine the singular symplectic quotients.
In Appendix \ref{alg-pre}, we present the algebro-geometric preliminaries of $\mathcal{C}^{\infty}$-ringed spaces.
In Appendix \ref{smo-sub}, we briefly review the definitions of smooth subcartesian structure and vector pseudobundle due to Aronszajn and Marshall, respectively.

\subsection*{Acknowledgements}
I am grateful to Professor Reyer Sjamaar and Professor Yi Lin for many helpful discussions on the topic of group actions.
I would like to express my great gratitude to Professor Guosong Zhao, Professor Xiaojun Chen, and Professor Bohui Chen for their constant encouragements and supports.
This work began when I was visiting the Department of Mathematics of Cornell University from Sept 2017 to Sept 2019, and most of this work was completed during my visits to the School of Mathematics of Sichuan University, and Tianyuan Mathemtical Center in Southwest China during the winters of 2020 and 2021.
I sincerely thank these institutes for hosting my research visits.
This work is partially supported by the National Nature Science Foundation of China (Grant No. 12271225 and 11701051).


\section{Preliminaries}\label{pre}
In this section we establish some notational conventions and present a brief review of stratification structures.
\subsection{Notation and conventions}\label{not-con}

Throughout this paper, \emph{topological spaces} are second countable and Hausdorff.
By a \emph{local ringed space} we shall mean a pair $(X,\mathcal{O}_{X})$ consisting of a topological space $X$ and a sheaf $\mathcal{O}_{X}$ of commutative local rings.
A $k$-\emph{space} is a local ringed space $(X,\mathcal{O}_{X})$ such that the structure sheaf $\mathcal{O}_{X}$ is a sheaf of algebras over $k$, where $k$ equals $\mathbb{R}$ or $\mathbb{C}$.

If $G$ is a compact connected Lie group with Lie algebra $\mathfrak{g}$, the complexification of $G$ is denoted by $G^{\mathbb{C}}$.
Given a symplectic manifold $(M,\omega)$ on which $G$ acts in a Hamiltonian fashion, the moment map from $M$ to $\mathfrak{g}^{\ast}$ is denoted by $\mu$.
The tuple $(M,\omega,G,\mu)$ is then called a \emph{symplectic Hamiltonian $G$-manifold}.
We denote by $M_{0}=\mu^{-1}(0)/G$ the \emph{symplectic quotient} at zero.
Assume that $J$ is a $G$-invariant complex structure on $M$.
By a \emph{K\"{a}hler Hamiltonian $G$-manifold} we mean a tuple $(M,ds^{2},G,\mu)$ such that:
\begin{itemize}
  \item [(i)] $ds^{2}$ is a $G$-invariant K\"{a}hler metric on $M$;
  \item [(ii)] $G^{\mathbb{C}}$ acts on $M$ holomorphically;
  \item [(iii)] $G$ acts on $(M, -\mathrm{Im}\,ds^{2})$ in a Hamiltonian fashion with the moment map $\mu$.
\end{itemize}

\subsection{Stratification structures}

The stratification structure naturally arises in the study of singular spaces such as algebraic varieties, analytic varieties, and singularities of smooth mappings.
Intuitively, the idea of a stratification is to decompose a topological space into a disjoint union of manifolds with different dimensions in a nice way.
Let us begin with the definition of a decomposition on a topological space.
\begin{defn}\label{decom}
Let $X$ be a paracompact Hausdorff topological space and $\mathscr{I}$ a partially ordered set with order relation denoted by $\leq$.
An \emph{$\mathscr{I}$-decomposition} of $X$ is a locally finite collection of disjoint locally closed subsets called pieces $S_{i}\subset X$ (one for each $i\in \mathscr{I})$ satisfies the following conditions:
\begin{enumerate}
  \item [(i)] $X=\coprod_{i\in \mathscr{I}}S_{i}$
        and each piece $S_{i}$ is a smooth manifold in the induced topology;
  \item [(ii)](Condition of frontier)
  $S_{i}\cap \overline{S}_{j}\neq\emptyset\Leftrightarrow S_{i}\subset \overline{S}_{j}\Leftrightarrow i\leq j$.
\end{enumerate}
\end{defn}

Set
$\mathcal{S}_{X}=\{S_{i}\subset X\mid i\in \mathscr{I}\}$.
The pair $(X,\mathcal{S}_{X})$ is called an \emph{$\mathscr{I}$-decomposed space}.
A piece $S_{i}\in \mathcal{S}_{X}$ is called a \emph{stratum} of $X$.
We say that a stratum $S\in\mathcal{S}_{X}$ is \emph{maximal}  (resp. \emph{minimal}) if it is open (resp. closed) in $X$.
The \emph{dimension} of a decomposed space $(X,\mathcal{S}_{X})$ is defined to be the dimension of the maximal stratum.
In general, a stratum $S\in \mathcal{S}_{X}$ is \emph{regular} or \emph{principal} if it is open in $X$, otherwise it is called a \emph{singular} stratum.
For any strata $S_{i}$ and $S_{j}$, we write $S_{i}\leq S_{j}$ if $S_{i}\subset\overline{S}_{j}$;
in particular, if $S_{i}\leq S_{j}$ and $S_{i}\neq S_{j}$
then we write $S_{i}< S_{j}$.

Let $(X,\mathcal{S}_{X})$ be a decomposed space.
The \emph{depth} of a stratum $S\in \mathcal{S}_{X}$ is defined by
\[
\mathrm{depth}_{X}(S):=\sup\{n\mid  S=S_{0}<S_{1}<\cdots<S_{n}\}
\]
where $S_{1},\cdots,S_{n}$ are strata of $X$.
The \emph{depth} of $(X,\mathcal{S}_{X})$ is defined to be
\[
\mathrm{depth}\,(X):=\sup_{i\in \mathscr{I}}\mathrm{depth}_{X}(S_{i}).
\]
Define $X^{i}=\bigcup_{j\leq i}S_{j}$, and $X^{i}$ is called a \emph{skeleton} of $X$.
Then there exists a finite filtration of skeletons
\[
X=X^{m}\supset X^{m-1}\supset\cdots\supset X^{0}\supset X^{-1}=\emptyset,
\]
where $m=\mathrm{depth}\,(X)$.
By definition, we have $m\leq \mathrm{dim}\,X$ and $X-X^{m-1}$ is dense in $X$.
To make the strata of a decomposition fit together in a nicer way, we have the following definition which is a recursion on the depth of the space.
\begin{defn}\label{t-l-tri}
A 0-dimensional stratified space $X$ is a discrete set of points with the trivial filtration
$X=X^{0}\supset X^{-1}=\emptyset$.

An $n$-dimensional \emph{stratified space} is a decomposed space $(X,\mathcal{S}_{X})$ satisfying the
\emph{topological local triviality}:

For each point $x$ in the stratum $S_{i}=X^{i}-X^{i-1}$ there exist an open neighborhood $U$ of $x$ in $X$, an open ball $\mathbb{B}^{i}$ of $x$ in $S_{i}$, a compact $(n-i-1)$-dimensional stratified space $L$, and a homeomorphism
$$
\phi:\mathbb{B}^{i}\times c(L)\longrightarrow U
$$
which takes each $\mathbb{B}^{i}\times c(L^{j-1})$ homeomorphically onto $X^{i+j}\cap U$.
Here $L$ is called the \emph{link} of $x$.
A decomposition $\mathcal{S}_{X}$ satisfying the topological local triviality is called a \emph{stratification} on $X$.
\end{defn}

\section{$\C^{\infty}$-stratified pseudomanifolds}\label{smo}
To study the differential geometry of singular spaces,
one needs to generalize the differential structure to the singular case.
In the spirit of smooth manifolds, Aronszajn \cite{Ar67} gave a definition of \emph{subcartesian space} in term of singular atlas, which was subsequently developed by Marshall \cite{Ma75} and
Aronszajn--Szeptycki \cite{AS80}.
A related notion was independently introduced by Spallek \cite{Sp69}.
In a more general setting, Sikorski \cite{Si67} introduced the notion of \emph{differential space}.
Paralleling Grothendieck's scheme theory, the theory of $\C^{\infty}$-differentiable spaces was developed in \cite{NS03}.
To the best of our knowledge, in the context of algebraic geometry, it was Hironaka \cite{Hi69} who first introduced the notion of \emph{differentiably stratified space} aiming at characterizing the equisingularity of algebraic varieties proposed by Zariski.
Following Aronszajn's subcartesian spaces, Pflaum presented a similar definition for stratified
spaces; moreover, under the viewpoints of analysis and geometry, he gave a systemical study of smooth stratified spaces in the monograph \cite{Pf01}.

\subsection{$\C^{\infty}$-structures on stratified spaces}
The purpose of this section is to present a brief review of $\C^{\infty}$-charts for stratified spaces introduced in \cite{Hi69, Pf01}.
\begin{defn}\label{sig-chart}
Let $(X,\mathcal{S}_{X})$ be a stratified (decomposed) space of dimension $n$ and $x$ an arbitrary point of $X$.
A \emph{$\C^{\infty}$-chart} around $x$ is a pair $(U,\phi_{U})$ satisfying the following conditions:
\begin{itemize}
  \item [(i)] $U$ is an open neighborhood of $x$;
  \item [(ii)] $\phi_{U}:U\rightarrow
      \mathrm{im}\,(\phi_{U})\subset\mathbb{R}^{m}$ $(m\geq n)$
      is a homeomorphism such that the image $\mathrm{im}\,(\phi_{U})$ is a locally closed subset of $\mathbb{R}^{m}$ and for each stratum $S\in\MC_{X}$ the restriction map
      $\phi_{U}:U\cap S\rightarrow\phi_{U}(U\cap S)$
      is a diffeomorphism.
\end{itemize}
In particular, if the assumption of locally closed property for $\mathrm{im}\,(\phi_{U})$ is dropped, then we call $(U,\phi_{U})$ a \emph{weak $\C^{\infty}$-chart}.
\end{defn}
In the spirit of smooth manifold, we have
\begin{defn}(\cite[Definition 1.3.3]{Pf01})\label{s-atlas}
Let $(X,\mathcal{S}_{X})$ be a stratified (decomposed) space.
If a given set of $\C^{\infty}$-singular charts
$$
\mathfrak{A}=
\{(U,\phi_{U}),(V,\phi_{V}),(W,\phi_{W}),\cdots\}
$$
on $X$ satisfies the following conditions, then we call $\mathfrak{A}$ a
\emph{$\C^{\infty}$-structure} on $X$:
\begin{itemize}
  \item [(i)] $\mathfrak{A}$ forms a $\C^{\infty}$-atlas on $X$ ;
  \item [(ii)] $\mathfrak{A}$ is \emph{maximal}, i.e., if a $\C^{\infty}$-singular chart $(U^{\prime},\phi^{\prime})$ is compatible with all coordinate charts in $\mathfrak{A}$ (in the sense of Definition \ref{ssp}), then $(U^{\prime},\phi^{\prime})\in\mathfrak{A}$.
\end{itemize}

If all $\C^{\infty}$-singular charts in $\mathfrak{A}$ are \emph{weak}, we say that $\mathfrak{A}$ is a \emph{weak $\C^{\infty}$-structure} on $X$.
\end{defn}

Now we introduce the notion of smooth stratified pseudomanifold.
\begin{defn}\label{s-s-pseodomfd}
We say that a stratified space $(X,\mathcal{S}_{X})$ is a \emph{smooth stratified pseudomanifold},
if there exists a $\C^{\infty}$-structure on $X$.
\end{defn}

We say that a stratified space $(X,\mathcal{S}_{X})$ is \emph{Euclidean embeddable}, if it admits a global $\C^{\infty}$-chart
$\phi:X\rightarrow\mathbb{R}^{N}$.
Akin to a smooth manifold, the condition (i) in Definition \ref{s-atlas} is primary.
It can be shown that if a set $\mathfrak{A}^{\prime}$ of $\C^{\infty}$-singular charts satisfies (i), then there is a unique $\C^{\infty}$-structure $\mathfrak{A}$ such that $\mathfrak{A}^{\prime}\subset\mathfrak{A}$.
Actually, denote by $\mathfrak{A}$ the set of all $\C^{\infty}$-singular charts which are compatible with every chart in $\mathfrak{A}^{\prime}$,
then $\mathfrak{A}$ is a $\C^{\infty}$-structure uniquely determined by $\mathfrak{A}^{\prime}$.
For this reason, to construct a $\C^{\infty}$-structure on $X$, we only need construct an open covering of $X$ by compatible $\C^{\infty}$-singular charts.

A continuous function $f:X\rightarrow\mathbb{R}$ is \emph{smooth}, if for any $\C^{\infty}$-chart $\phi:U\rightarrow\mathbb{R}^{m}$ there is a smooth function $F$ from $\mathbb{R}^{m}$ to $\mathbb{R}$ such that $f|_{U}=F\circ\phi$.
We denote by $\C^{\infty}_{X}$ the sheaf of smooth functions on $X$ and call it the \emph{structure sheaf} of $(X,\mathcal{S}_{X})$.
For any $x\in X$, let $\mathfrak{m}_{x}$ be the ideal of germs of smooth functions vanishing at $x$.
It is noteworthy that $\mathfrak{m}_{x}$ is the unique maximal ideal of the stalk $\C^{\infty}_{X,x}$ and therefore $(X,\C^{\infty}_{X})$ forms a \emph{reduced local $\mathcal{C}^{\infty}$-ringed space} (see Definition \ref{inf-r-sps}).
For the sake of brevity, we also say that a triple $\mathcal{X}=(X,\mathcal{S}_{X},\C^{\infty}_{X})$
is a smooth stratified pseudomanifold.
In particular, we have the following
\begin{thm}\label{fine}
The structure sheaf $\C^{\infty}_{X}$ of a smooth stratified pseudomanifold $\mathcal{X}$ is a fine sheaf.
\end{thm}
\begin{proof}
See the proof of \cite[Theorem 1.3.13]{Pf01} or the proof of \cite[Proposition 1.2]{Ma75} in a more general setting.
\end{proof}

We now consider the morphism between smooth stratified pseudomanifolds.
\begin{defn}\label{str-sm-mp}
Let $\mathcal{X}=(X,\mathcal{S}_{X},\C^{\infty}_{X})$
and
$\mathcal{Y}=(Y,\mathcal{S}_{Y},\C^{\infty}_{Y})$
be two smooth stratified pseudomanifolds.
A continuous map $f:X\rightarrow Y$ is called a \emph{smooth map} from $\mathcal{X}$ to $\mathcal{Y}$, if it satisfies the following conditions:
\begin{itemize}
  \item [(i)] $f$ is stratum-preserving;
   i.e., it maps a stratum of $\mathcal{S}_{X}$ to a stratum of $\mathcal{S}_{Y}$;
  \item [(ii)] for any open subset $V$ in $Y$ and any smooth function $g\in\C^{\infty}_{Y}(V)$ the function $g\circ f$ is an element of $\C^{\infty}_{X}(f^{-1}(V))$.
\end{itemize}
\end{defn}

It is noteworthy that smooth stratified pseudomanifolds with smooth maps form a category which contains the category of smooth manifolds as a full subcategory.
In particular, let $Y$ be a \emph{smooth manifold}.
A \emph{proper embedding} of $\mathcal{X}$ into the smooth manifold $Y$ is a smooth, injective, and proper map $f:X\rightarrow Y$ such that the pullback of smooth functions
$f^{\ast}:\C^{\infty}(Y)
\rightarrow\C^{\infty}(X)$
is surjective.
From a $\C^{\infty}$-algebro-geometric viewpoint, a smooth stratified pseudomanifold is a reduced local $\mathcal{C}^{\infty}$-ringed space in the sense of Definition \ref{inf-r-sps}.
Observe that a smooth stratified map
$f:\mathcal{X}\rightarrow\mathcal{Y}$
determines a morphism of sheaves of $\mathcal{C}^{\infty}$-rings denoted by
$f_{\sharp}:\C^{\infty}_{Y}\rightarrow f_{*}\C^{\infty}_{X}$.
As a result, the pair $\underline{f}=(f,f_{\sharp})$ becomes a morphism of local $\mathcal{C}^{\infty}$-ringed spaces.


To conclude this subsection, we state some examples of smooth stratified pseudomanifolds, which appear naturally in differential and algebraic geometries.
\begin{ex}[Smooth $G$-manifolds]\label{G-mfd}
Let $M$ a smooth manifold on which $G$ acts smoothly and effectively.
For each closed subgroup $H$ of $G$ we can define a subspace $M_{(H)}$ which contains all points of $M$ such that the stabilizers are conjugate to $H$.
In general, $M_{(H)}$ is not a submanifold; however, its connected components are submanifolds with different dimensions.
Consequently, we can decompose $M$ as
\begin{equation*}\label{orbit-type}
M=\coprod_{H<G}M_{(H)}
\end{equation*}
which yields a decomposition $\mathcal{S}_{\mathrm{orb}}$ on $M$ and makes it into a Whitney stratified space (cf. \cite[Theorem 4.3.7]{Pf01}).
Suppose that
$$\mathfrak{A}=
\{(U_{\lambda},\phi_{\lambda})\,|\,\lambda\in\Lambda\}$$
is a $\C^{\infty}$-atlas of $M$.
Then, from definition, $\mathfrak{A}$ serves as $\C^{\infty}$-singular atlas of the decomposed space $(M,\mathcal{S}_{\mathrm{orb}})$ in the sense of Definition \ref{s-atlas}.
Observe that $(M,\mathcal{S}_{\mathrm{orb}})$ is a Whitney stratified space.
The topological local triviality holds necessarily.
It follows that $(M,\mathcal{S}_{\mathrm{orb}},\C^{\infty}_{M})$ is a smooth stratified pseudomanifold in the sense of Definition \ref{s-s-pseodomfd}.

Consider the orbit space $M/G$.
Let $\pi:M\rightarrow M/G$ be the quotient map.
There exists a natural decomposition $\mathfrak{S}_{M/G}$:
\begin{equation*}\label{orbit-type-1}
M/G=\coprod_{H<G}M_{(H)}/G
\end{equation*}
satisfying the Whitney's Condition (B) (cf. \cite[Theorem 4.4.6]{Pf01}).
Moreover, the smooth structure on $M$ induces a canonical smooth structure on $M/G$ given by
$$
\C^{\infty}_{M/G}(U)
=\C^{\infty}_{M}(\pi^{-1}(U))^{G},
$$
where $U$ is an open subset of $M/G$ under the quotient topology.
This implies that the orbit space $M/G$ is also a smooth stratified pseudomanifold.
\end{ex}

\begin{ex}[Symplectic quotients]\label{sym-quot}
Let $(M,\omega,G,\mu)$ be a compact Hamiltonian $G$-manifold.
If 0 is a \emph{regular} value of the moment map $\mu$, then the level set $Z=\mu^{-1}(0)$ is a closed submanifold of $M$ with a locally free $G$-action.
The \emph{symplectic quotient} $M_{0}=Z/G$ is an orbifold and we have the \emph{inclusion-quotient diagram:}
$$
\xymatrix{
  Z \ar[d]_{\pi} \ar[r]^{\imath} &   M     \\
  M_{0}                     }
$$
where $\pi$ is the orbit map and $\imath$ is the inclusion map.
Since the $G$-action on $Z$ is locally free the symplectic quotient $M_{0}$ is an orbifold.
According to the symplectic reduction procedure, there exists a unique non-degenerate 2-form $\omega_{0}$ on $M_{0}$ such that $\pi^{\ast}\omega_{0}=\imath^{\ast}\omega$.
In general, without the regularity assumption the space $M_{0}$ is a \emph{stratified symplectic space} in the sense of Sjamaar--Lerman \cite{SL91} and with a canonical decomposition given by
$$
M_{0}=\coprod_{H<G}\bigl(M_{(H)}\cap Z\bigr)/G.
$$
In addition, $M_{0}$ admits a natural $\C^{\infty}$-singular atlas and can be embedded into some Euclidean space $\mathbb{R}^{N}$ as a \emph{Whitney stratified space}, see \cite[Section 6]{Pf01b} and \cite[Theorem 6.7]{SL91}.
So $M_{0}$ becomes a smooth stratified pseudomanifold.
\end{ex}

\begin{ex}[Complex analytic varieties]\label{c-a-v}
Let $M$ be a complex manifold with complex dimension $n$ and $V$ an analytic subvariety of $M$.
Set $V^{\ast}$ the locus of smooth points of $V$ and $V_{s}=V-V^{\ast}$ the locus of singular points.
It is worth noting that $V^{\ast}$ is a complex submanifold and $V_{s}$ is an analytic subvariety of $M$.
In general, given an analytic subvariety $V$ we may split it as
$V=V^{\ast}\cup V_{s}$.
Also, we may split $V_{s}$ into $(V_{s})^{\ast}$ and $(V_{s})_{s}$, etc.
As a result, setting $V_{1}=V^{\ast}$, $V_{2}=(V_{s})^{\ast}$, $V_{3}=((V_{s})_{s})^{\ast}$, etc. yields a partition
$$
V=V_{1}\cup V_{2}\cup V_{3}\cdots,
$$
where $V_{i}$ are complex submanifolds of $M$ satisfying
$$
\mathrm{dim}\,V=\mathrm{dim}\,V_{1}>\mathrm{dim}\,V_{2}>
\mathrm{dim}\,V_{3}>\cdots.
$$
Via splitting $V_{i}$ above into its connected components,
we get a refined partition which satisfies the condition of frontier in Definition \ref{decom} and therefore we get a \emph{decomposition} $\mathcal{S}_{V}$ of $V$ such that the Whitney's Condition (B) holds (cf. \cite[Theorem 19.2]{Wh65a}).
Given a $\C^{\infty}$-atlas of the ambient manifold
$$
\mathfrak{A}=
\bigl\{\phi_{\lambda}:U_{\lambda}\rightarrow
\mathbb{R}^{2n}\,\big|\,
\lambda\in\Lambda\bigr\}.
$$
Then for each $\lambda\in\Lambda$, the restriction
$
\tilde{\phi}_{\lambda}:V\cap U_{\lambda}\rightarrow
\mathbb{R}^{2n}
$
comprises a $\C^{\infty}$-singular chart on $V$.
As a result, the variety $V$ admits a natural $\C^{\infty}$-singular atlas inherited from $M$.
It follows that $(V,\mathcal{S}_{V},\C^{\infty}_{V})$ is a smooth stratified pseudomanifold.
In particular, when $M$ is the complex projective space $\mathbb{CP}^{n}$ we get that each projective variety $V$ in $\mathbb{CP}^{n}$ has a natural structure of smooth stratified pseudomanifold.
The smooth structure $\C^{\infty}_{V}$ described above is the smallest smooth structure on $V$ such that the inclusion map $i:V\hookrightarrow M$ is smooth.
\end{ex}

\subsection{Tangent pseudobundles}\label{t-ct-bundle}

Let $\mathcal{X}=(X,\mathcal{S}_{X},\C^{\infty}_{X})$
be a smooth stratified pseudomanifold with a fixed $\C^{\infty}$-atlas
$$
\mathfrak{A}_{X}=
\bigl\{\phi_{\lambda}:U_{\lambda}\rightarrow
\mathbb{R}^{n_{\lambda}}\,\big|\,\lambda\in\Lambda\bigr\}.
$$
Set-theoretically, one can define the tangent bundle of $X$ to be
\begin{equation*}\label{set-tx}
TX=\coprod_{S\in\MC_{X}}TS
\end{equation*}
together with a canonical footpoint projection
$$
\pi:TX\longrightarrow X
$$
by setting
$\pi|_{TS}:TS\rightarrow S$ to be the projection of the tangent bundle for $S$.
In the spirit of smooth vector bundle, the first thing is to assign a suitable topology to $TX$ such that $\pi$ becomes \emph{continuous}.
This can be realized as follows.

For any open subset $U$ of $X$, we define the set
$TU=\pi^{-1}(U)$.
Define the map
\begin{equation}\label{t-phi-lam}
T(\phi_{\lambda}):
TU_{\lambda}\longrightarrow
T\mathbb{R}^{n_{\lambda}}\cong\mathbb{R}^{2n_{\lambda}}
\end{equation}
by requiring
$
T(\phi_{\lambda})|_{TS\cap TU_{\lambda}}=
T(\phi_{\lambda}|_{S\cap U_{\lambda}})
$
for any stratum $S\in\MC_{X}$.
Endow $TX$ the \emph{coarsest topology} on $TX$ such that the set $TU$ is \emph{open} in $TX$ and the map \eqref{t-phi-lam} is \emph{continuous}.
Under this topology, the projection $\pi:TX\rightarrow X$ becomes continuous;
furthermore, for every stratum $S\in\MC_{X}$, the set $TS$ becomes a locally closed subset in $TX$ and $\pi$ restricts to a smooth map on $TS$.

In general, $TX$ is not a vector bundle unless $X$ is a smooth manifold;
however, the following result shows that $TX$ admits a natural structure of smooth vector pseudobundle.
\begin{thm}\label{pseu-v-b}
The $\C^{\infty}$-atlas $\mathfrak{A}_{X}$ induces a $\C^{\infty}$-subcartesian structure on $TX$ (see Definition \ref{ssp}).
Moreover, the triple $(TX,\pi,X)$ is a $\C^{\infty}$-vector pseudobundle (see Definition \ref{v-p-bundle}).
\end{thm}
\begin{proof}
From definition, the collection
$$
\mathfrak{A}_{TX}=\bigl\{T\phi_{\lambda}:TU_{\lambda}
\rightarrow
\mathbb{R}^{2n_{\lambda}}\,|\,\lambda\in\Lambda\bigr\}
$$
comprises a set of $\C^{\infty}$-charts for $TX$.
It remains to verify the elements in $\mathfrak{A}_{TX}$ are compatible with each others.
The proof is the same as that in the proof of \cite[Theorem 2.1.2]{Pf01}.
It turns out that $TX$ becomes a $\C^{\infty}$-subcartesian space with the $\C^{\infty}$-atlas $\mathfrak{A}_{TX}$ and $\pi:TX\rightarrow X$ becomes a smooth surjection.
For each $x\in X$, by definition, there holds $\pi^{-1}(x)=T_{x}S$, where $S$ is the stratum containing the point $x$.
A direct checking shows that the vector operators:
$$
+:TX\times_{X}TX
\longrightarrow TX\,\,\,\,
\mathrm{and}\,\,\,\,
\cdot:\mathbb{R}\times TX\longrightarrow TX
$$
are smooth maps and therefore $(TX,\pi,X)$ is a $\C^{\infty}$-family of $\mathbb{R}$-vector spaces.
On the one hand, for every $\lambda\in\Lambda$ we have
$$
TU_{\lambda}=\pi^{-1}(U_{\lambda})=\pi^{-1}(\pi(TU_{\lambda}))
$$
since $\pi(TU_{\lambda})=U_{\lambda}$.
On the other hand, we get a commutative diagram of smooth maps:
\begin{equation}\label{fml-mp}
\vcenter{
\xymatrix@=1.3cm{
  TU_{\lambda} \ar[d]_{\pi} \ar[r]^{T\phi_{\lambda}} & \mathbb{R}^{2n_{\lambda}} \ar[d]^{\Pi} \\
  U_{\lambda} \ar[r]^{\phi_{\lambda}} & \mathbb{R}^{n_{\lambda}},   }
  }
\end{equation}
where $\Pi$ is the projection via the first $n_{\lambda}$ coordinates.
Note that $T\phi_{\lambda}$ is $\mathbb{R}$-linear along the fibers.
It follows from the definition that the diagram \eqref{fml-mp} is a morphism of $\C^{\infty}$-families of $\mathbb{R}$-vector spaces.
Thus, by Definition \ref{v-p-bundle}, we are led to the conclusion that $(TX,\pi,X)$ is a $\C^{\infty}$-vector pseudobundle.
\end{proof}

By a \emph{smooth vector field} on $X$ we mean a smooth map $V:X\rightarrow TX$ with $\pi\circ V=\mathrm{id}_{X}$.
We denote by $\mathrm{Vect}_{\mathrm{str}}(X)$
the set of smooth vector fields on $X$.
For any $f\in\C^{\infty}(X)$ and $V\in\mathrm{Vect}_{\mathrm{str}}(X)$, we can define a new function $V(f)$ by requiring
$$
\bigl(V(f)\bigr)|_{S}=V|_{S}(f|_{S}),
$$
for every stratum $S\in\MC_{X}$.
Likewise, given two smooth vector fields $V$ and $W$ on $X$, we can
define $[V,W]_{\mathrm{str}}$ by setting
$$
\bigl([V,W]_{\mathrm{str}}\bigr)|_{S}
=[V|_{S},W|_{S}].
$$

\begin{lem}\label{str-lie}
With the same situation as above, we have
\begin{itemize}
  \item [(i)] $V(f)$ is a smooth function on $X$;
  \item [(ii)] The bracket $[-,-]_{\mathrm{str}}$ is a Lie bracket on $\mathrm{Vect}_{\mathrm{str}}(X)$.
\end{itemize}
\end{lem}
\begin{proof}
For a proof, see the arguments in \cite[\S\,2.2.5]{Pf01} and \cite[\S\,2.2.7]{Pf01}.
\end{proof}

We call $[-,-]_{\mathrm{str}}$ the \emph{stratified Lie bracket} on $X$.
A theorem of Cushman--\'{S}niatycki \cite[Theorem 3.3]{CS20} says that on a smooth subcartesian space derivations in the usual sense are $\C^{\infty}$-derivations, see Appendix \ref{alg-pre}.
Note that every smooth stratified pseudomanifold is a smooth subcartesian space in the sense of Definition \ref{ssp}.
This means every smooth vector field on a smooth stratified pseudomanifold is a $\C^{\infty}$-derivation.
It is noteworthy that $\mathrm{Vect}_{\mathrm{str}}(X)$ does not coincide with the space of $\C^{\infty}$-derivations in general.
However, we always have the following relation as Lie algebras:
$$
\mathrm{Vect}_{\mathrm{str}}(X)\subset
\mathrm{Der}_{\C^{\infty}(X)}(\C^{\infty}(X),\C^{\infty}(X))
\cong\mathcal{T}_{X}(X).
$$

Comparing to the tangent bundle of a smooth manifold,
one may wonder whether $TX$ is a smooth decomposed space and $\pi$ is a topological projection, i.e., $X$ carries the quotient topology of $TX$.
In general, the answer is negative unless $X$ satisfies the Whitney's Condition (A) (cf.  \cite[Theorem 2.1.2]{Pf01}).

\begin{rem}
Likewise, as a set, one can define the \emph{cotangent bundle} to be
\begin{equation}\label{c-t-x}
T^{\ast}X=\coprod_{S\in\MC_{X}}T^{\ast}S
\end{equation}
with a canonical projection
$\pi:T^{\ast}X\rightarrow X$
such that, for each stratum $S\in\MC_{X}$, the restriction
$$
\pi|_{T^{\ast}S}:T^{\ast}S\longrightarrow S
$$
is the projection of the cotangent bundle of $S$.
In contrast to the tangent pseudobundle,
even if the Whitney's Condition (A) is satisfied,
the smooth structure on $(X,\mathcal{S}_{X})$ can only induces a canonical topology on $T^{\ast}X$ such that $\pi$ is \emph{continuous}, and $T^{\ast}X$ becomes a \emph{decomposed space} with respect to the decomposition
\eqref{c-t-x}, see \cite[\S\,2.3.3]{Pf01} for a detailed construction.
In fact, $T^{\ast}X$ is a \emph{$\C^{0}$-family of $\mathbb{R}$-vector spaces} over $\mathcal{X}$ in the sense of Definition \ref{v-p-bundle}.
\end{rem}
\subsection{$\mathcal{C}^{\infty}$-algebraic de Rham cohomology}\label{app-a2}

This subsection aims to a cohomological study of smooth stratified pseudomanifolds.
We keep the notational conventions of the subsection above.
Let $\mathcal{X}=(X,\mathcal{S}_{X},\C^{\infty}_{X})$
be a smooth stratified pseudomanifold.
In contrast to smooth manifolds, the object $T^{\ast}X$ does not admit any smooth structure in a obvious way.
This implies that $T^{\ast}X$ and its exterior products $\wedge^{k}T^{\ast}X$ are not fit to define \lq\lq differential forms" on $X$.
However, from a $\mathcal{C}^{\infty}$-algebro-geometric point of view, as a local $\mathcal{C}^{\infty}$-ringed space $(X,\C^{\infty}_{X})$, the cotangent sheaf and the sheaves of $\mathcal{C}^{\infty}$-K\"{a}hler differential forms contain the most fundamental geometric information of $X$.

According to the de Rham cohomology theory of local $\mathcal{C}^{\infty}$-ringed spaces in Appendix \ref{alg-pre},
we get a sheaf complex over $X$:
\begin{equation}\label{fine-de-com-X}
\xymatrix{
  0 \ar[r] & \C^{\infty}_{X} \ar[r]^{\mathrm{d}} & \Omega^{1}_{X} \ar[r]^{\mathrm{d}} & \Omega^{2}_{X} \ar[r]^{\mathrm{d}} & \cdots \ar[r]^{\mathrm{d}} & \Omega^{k}_{X} \ar[r] & \cdots. }
\end{equation}
Note that the structure sheaf $\C^{\infty}_{X}$ is \emph{fine} (Theorem \ref{fine}).
The sheaf of $\mathcal{C}^{\infty}$-K\"{a}hler differential $p$-forms $\Omega^{p}_{X}$ is also fine since it is a sheaf of $\C^{\infty}_{X}$-modules.
As a result, applying the global section functor to \eqref{fine-de-com-X} gives rise to the \emph{de Rham complex of $\mathcal{C}^{\infty}$-K\"{a}hler differential forms}:
\begin{equation}\label{dR-com-X}
\xymatrix{
  0 \ar[r] & \C^{\infty}(X) \ar[r]^{\mathrm{d}} & \Omega^{1}(X) \ar[r]^{\mathrm{d}} & \Omega^{2}(X) \ar[r]^{\mathrm{d}} & \cdots \ar[r]^{\mathrm{d}} & \Omega^{k}(X) \ar[r] & \cdots. }
\end{equation}
Denote the cohomology of \eqref{dR-com-X} by
$\mathrm{H}^{\ast}_{\DR}(X)$.
Recall the definition of $\mathcal{C}^{\infty}$-algebraic de Rham cohomology of local $\mathcal{C}^{\infty}$-ringed spaces in Appendix \ref{alg-pre}.
Thanks to the \emph{fineness} of the terms in
\eqref{fine-de-com-X}, the hypercohomology of \eqref{fine-de-com-X} is isomorphic to the cohomology of the complex \eqref{dR-com-X} and therefore it follows
$$
\mathbb{H}^{\ast}(X,\Omega^{\bullet}_{X})\cong
\mathrm{H}^{\ast}(\Omega^{\bullet}(X),\mathrm{d})=\mathrm{H}^{\ast}_{\DR}(X).
$$
It is worth pointing out that the sheaf complex \eqref{fine-de-com-X} is not exact and hence the de Rham theorem does holds on $X$, that is the \v{C}ech cohomology of the constant sheaf $\check{\mathrm{H}}^{\ast}(X, \underline{\mathbb{R}})$ is not isomorphic to the $\mathcal{C}^{\infty}$-algebraic de Rham cohomology $\mathrm{H}^{\ast}_{\DR}(X)$.
Let $\underline{f}=(f,f_{\sharp})$ be a smooth map from $\mathcal{X}$ to another smooth stratified pseudomanifold $\mathcal{Y}=(Y,\mathcal{S}_{Y},\C^{\infty}_{Y})$ in the sense of Definition \ref{str-sm-mp}.
According to \cite[Theorem 8.13]{Ler22}, the map $\underline{f}$ induces a morphism of the complexes
$$
\underline{f}^{\ast}:(\Omega^{\bullet}(Y),\mathrm{d})\longrightarrow (\Omega^{\bullet}(X),\mathrm{d})
$$
and hence a morphism of cohomology groups
\begin{equation}\label{com-map-3}
\underline{f}^{\ast}:\mathrm{H}^{\ast}_{\DR}(Y).
\longrightarrow
\mathrm{H}^{\ast}_{\DR}(X)
\end{equation}
This implies that the $\mathcal{C}^{\infty}$-algebraic de Rham cohomology is a \emph{contravariant functor} from the category of smooth stratified pseudomanifolds to the category of abelian groups.

Compare to the differential forms on smooth manifolds.
We have the following local result.

\begin{lem}\label{local-finite}
Let $\mathcal{X}=(X,\mathcal{S}_{X},\C^{\infty}_{X})$
be a smooth stratified pseudomanifold.
The cotangent sheaf $\Omega^{1}_{X}$ is locally finitely generated.
\end{lem}
\begin{proof}
For every $x$ in $X$, let $\phi:U\rightarrow O\subset\mathbb{R}^{n}$ be a $\C^{\infty}$-chart around $x$.
Then the $\mathbb{R}$-algebras $\C^{\infty}(U)$ and $\C^{\infty}(O)/\mathcal{I}$ are identical,
where $\mathcal{I}$ is the ideal of smooth functions vanishing on the image $\phi(U)$.
This implies that $\C^{\infty}(U)$ is a finitely generated $\C^{\infty}$-ring in the sense of Definition \ref{f-g-c-r}.
Then it follows from \cite[Proposition 5.6]{Jo19} that the cotangent module of $\C^{\infty}(U)$ is a finitely generated $\C^{\infty}(U)$-module and therefore the cotangent sheaf $\Omega^{1}_{X}$ is locally finitely generated as a $\C^{\infty}_{X}$-module.
\end{proof}

From Lemma \ref{local-finite}, we get that the stalk $\Omega^{1}_{X,x}$ is a finitely generated module over the local ring $\C^{\infty}_{X,x}$.
It is important to point out that we can not identify $\Omega^{1}_{X}$ with the sheaf of sections of some vector bundle over $X$ unless $X$ is a smooth manifold.
The reason lies in the fact that, in general, the sheaf $\Omega^{1}_{X}$ is \emph{not} locally free in the presence of singular strata.

To conclude this subsection, we prove the following two basic propositions that we will use in the sequel.

\begin{prop}\label{X-S}
Let $\mathcal{X}=(X,\mathcal{S}_{X},\C^{\infty}_{X})$
be a smooth stratified pseudomanifold and $S$ a stratum of $X$.
Then each closed $\C^{\infty}$-K\"{a}hler differential $p$-form $\alpha\in\Omega^{p}(X)$ induces a closed differential $p$-form on $S$ in the usual sense.
\end{prop}
\begin{proof}
Note that the stratum $S$ is a smooth manifold.
Consider $S$ as a local $\C^{\infty}$-ringed space $(S,\C^{\infty}_{S})$.
The inclusion $i:S\subset X$ induces a natural morphism of local $\C^{\infty}$-ringed spaces
$$
\underline{i}=(i,i_{\sharp}):(S,\C^{\infty}_{S})\longrightarrow
(X,\C^{\infty}_{X})
$$
and therefore a morphism of $\C^{\infty}$-K\"{a}hler differential $p$-forms
$$
\underline{i}^{\ast}:\Omega^{p}(X)
\longrightarrow
\Omega^{p}(S).
$$
Since $S$ is a smooth manifold, the space of $\C^{\infty}$-K\"{a}hler differential forms together with the operator $\mathrm{d}$ on $S$ coincide with the space of differential forms and the exterior differential operator in the usual sense.
It follows that the pullback $\underline{i}^{\ast}\alpha$ is a  differential $p$-form on $S$ in the usual sense.
The fact that $\mathrm{d}$ commutates with the map $\underline{i}^{\ast}$ implies that $\underline{i}^{\ast}\alpha$ is closed and this completes the proof.
\end{proof}

\begin{prop}\label{M-X}
Let $M$ be a smooth manifold and $X$ a closed subset in $M$.
Suppose $(X,\mathcal{S}_{X})$ is a stratified pseudomanifold with a smooth structure $\C^{\infty}_{X}$
inherited from $M$.
Then for each closed differential form $\alpha\in\Omega^{p}(M)$ the pullback $\underline{i}^{\ast}\alpha$ is a closed $\C^{\infty}$-K\"{a}hler differential $p$-form on $X$, where $\underline{i}=(i,i_{\sharp}):(X,\C^{\infty}_{X})\hookrightarrow (M,\C^{\infty}_{M})$ is the inclusion.
\end{prop}

The proof of Proposition \ref{M-X} is similar to Proposition \ref{X-S}.
On the one hand, as a local $\C^{\infty}$-ringed space, the sheaf complex of $\C^{\infty}$-K\"{a}hler differential forms on $M$ is equal to the sheaf complex of ordinary differential forms.
On the other hand, since the smooth structure of $X$ is inherited from $(M,\C^{\infty}_{M})$ there exists a natural morphism of local $\C^{\infty}$-ringed spaces
$$
\underline{i}=(i,i_{\sharp}):(X,\C^{\infty}_{X})
\longrightarrow
(M,\C^{\infty}_{M}).
$$
The rest of the proof is the same as Proposition \ref{X-S}.

\section{Symplectic and K\"{a}hler structures}\label{sym-kah}

\subsection{Symplectic stratified pseudomanifolds}

Let $\mathcal{X}=(X,\mathcal{S}_{X},\C^{\infty}_{X})$
be a $\C^{\infty}$-stratified pseudomanifold.
Recall the construction in Subsection \ref{app-a2}.
It has been shown that the sheaves of $\C^{\infty}$-K\"{a}hler differential forms on $X$ have many advantageous properties.
However, these sheaves do not behave well near singular points.
On the one hand, there exist non-zero $\C^{\infty}$-K\"{a}hler differential forms such that the restrictions to the regular stratum $X_{\mathrm{reg}}$ vanish.
On the other hand, in general $\Omega^{p}_{X}(X)\neq0$ when $p>\mathrm{dim}\,X$.
Inspired by \cite[Definition (1.1)]{Fer70} and \cite[Definition 2.3]{HKK17}, we are therefore led to introduce the following definition.
\begin{defn}\label{tor-ele}
Let $\mathcal{F}$ be a sheaf of $\C^{\infty}_{X}$-modules on $X$ and $U$ an arbitrary open subset of $X$.
We say that a section $\alpha\in\mathcal{F}(U)$ is a \emph{torsion element}, if the restriction $\alpha|_{U_{\mathrm{reg}}}$ is vanishing, where $U_{\mathrm{reg}}=U\cap X_{\mathrm{reg}}$.
If $\mathcal{F}_{\mathrm{t}}(U)=0$ for any open subset $U$ of $X$, we call $\mathcal{F}$ a \emph{torsion-free} sheaf,
where $\mathcal{F}_{\mathrm{t}}(U)$ is the $\C^{\infty}_{X}(U)$-submodule of torsion elements in $\mathcal{F}(U)$.
\end{defn}

\begin{rem}
Generally, it is noteworthy that the torsion element in the sense of Definition \ref{tor-ele} and the torsion element in the module theoretic sense are not identical.
\end{rem}

Set $\mathcal{F}_{\mathrm{tf}}(U)$ the quotient module of $\mathcal{F}_{\mathrm{t}}(U)$ in $\mathcal{F}(U)$.
This defines a presheaf on $X$ by assigning the $\C^{\infty}_{X}(U)$-module $\mathcal{F}_{\mathrm{tf}}(U)$ to an arbitrary open subset $U\subset X$.
The sheafification of this presheaf, denoted by $\mathcal{F}_{\mathrm{tf}}$, is called the \emph{torsion-free sheaf} corresponding to $\mathcal{F}$.
Note that the structure sheaf $\C^{\infty}_{X}$ itself is torsion-free and the cotangent sheaf $\Omega^{1}_{X}$ is locally finitely generated.
It is of importance to point out that $\Omega^{1}_{X}$ is \textbf{not} torsion-free in general.

Consider the tangent sheaf
$\mathcal{T}_{X}=\Hom_{\C_{X}^{\infty}}(\Omega^{1}_{X},\C_{X}^{\infty}).$
Since $\C^{\infty}_{X}(U)$ is not noetherian the sheaf $\mathcal{T}_{X}$ is not locally finitely generated necessarily.
\begin{prop}
The tangent sheaf $\mathcal{T}_{X}$ is torsion-free in the sense of Definition \ref{tor-ele}.
\end{prop}
\begin{proof}
Recall that the tangent sheaf $\mathcal{T}_{X}$ is the sheafification of the presheaf
\begin{equation*}
U\longmapsto
\Hom_{\C^{\infty}(U)}\bigl(\Omega^{1}_{X}(U),\C^{\infty}(U)\bigr),
\end{equation*}
where $U$ is an arbitrary open subset in $X$.
To show that $\mathcal{T}_{X}$ is torsion-free,
it is sufficient to prove that, as a $\C^{\infty}(U)$-module,
$$
\Hom_{\C^{\infty}(U)}\bigl(\Omega^{1}_{X}(U),\C^{\infty}(U)\bigr)
$$
is torsion-free.
For each element $s\in\Hom_{\C^{\infty}(U)}\bigl(\Omega^{1}_{X}(U),\C^{\infty}(U)\bigr)$, we get a commutative square:
\begin{equation}\label{tgt-tf}
\vcenter{
\xymatrix@=1.3cm{
  \Omega^{1}_{X}(U) \ar[d]_{\rho^{1}_{\mathrm{reg}}} \ar[r]^{s} & \C^{\infty}(U) \ar[d]^{\rho^{\infty}_{\mathrm{reg}}} \\
  \Omega^{1}_{X}(U_{\mathrm{reg}}) \ar[r]^{s_{\mathrm{reg}}} & \C^{\infty}(U_{\mathrm{reg}}).}
  }
\end{equation}
Here both $\rho^{1}_{\mathrm{reg}}$ and $\rho^{\infty}_{\mathrm{reg}}$ are restriction maps, and $s_{\mathrm{reg}}$ is the restriction of $s$ to $U_{\mathrm{reg}}$.
Note that $\rho^{\infty}_{\mathrm{reg}}$ is injective since $\C^{\infty}_{X}$ is torsion-free.
If $s_{\mathrm{reg}}=0$ then it follows from the commutativity of \eqref{tgt-tf} that $s$ equals zero, and this implies that $\mathcal{T}_{X}$ is torsion-free.
\end{proof}

For any section $v\in\mathcal{T}_{X}(U)$ and any $\C^{\infty}$-K\"{a}hler differential 1-form $\alpha\in\Omega_{X}^{1}(U)$,
we can insert $v$ in $\alpha$ to get a smooth function denoted by $\iota(v)\alpha\in\C^{\infty}(U)$.
In fact, we can generalize the classical contraction in differential geometry to $\C^{\infty}$-algebraic setting without essential changes.
Let $\beta$ be a $\C^{\infty}$-K\"{a}hler differential $p$-form on $U$.
For all $x\in U$, the germ of $\beta$ at $x$, denoted by $\beta_{x}$, lies in the $\C^{\infty}_{X,x}$-module
$\Omega^{p}_{X,x}=\wedge^{p}\Omega^{1}_{X,x}$ and therefore $\beta_{x}$ has a representation of the form
$$
\beta_{x}=
\alpha_{1}\wedge\alpha_{2}\wedge\cdots\wedge\alpha_{p}
$$
where $\alpha_{j}\in\Omega^{1}_{X,x}$ for any $1\leq j\leq p$.
The \emph{contraction}
$$
\iota(v):\Omega_{X}^{p}(U)\longrightarrow\Omega_{X}^{p-1}(U)
$$
is defined by requiring
$$
(\iota(v)\beta)_{x}=\sum_{j=1}^{p}(-1)^{j-1}
(\iota(v_{x})\alpha_{j})\cdot\alpha_{1}\wedge\cdots
\wedge\widehat{\alpha}_{j}\wedge\cdots\wedge\alpha_{p},
$$
where the hat indicates that $\alpha_{j}$ is omitted.
Having shown the contraction, it is natural to define the Lie derivative in the algebraic setting.
Recalling the classical Cartan formula in differential geometry, we can take the formula
\begin{equation}\label{li-der}
\mathcal{L}(v)=\iota(v)\circ\mathrm{d}+\mathrm{d}\circ\iota(v)
\end{equation}
as a definition of the \emph{Lie derivative} acting on $\C^{\infty}$-K\"{a}hler differential forms.
In particular, it satisfies the following:
\begin{prop}\label{cartan-fomla}
For any derivations $v\in\mathcal{T}_{X}(X)$ and $w\in\mathcal{T}_{X}(X)$, we have
\begin{itemize}
  \item [(i)] $[\mathcal{L}(v),\mathrm{d}]=0$;
  \item [(ii)] $[\mathcal{L}(v),\iota(w)]=\iota([v,w])$;
  \item [(iii)] $[\mathcal{L}(v),\mathcal{L}(w)]=\mathcal{L}([v,w])$.
\end{itemize}
\end{prop}
\begin{proof}
From \eqref{li-der}, we have
$
[\mathcal{L}(v),\mathrm{d}]=\mathrm{d}\circ\iota(v)\circ\mathrm{d}
-\mathrm{d}\circ\iota(v)\circ\mathrm{d}=0
$
since $\mathrm{d}\circ\mathrm{d}=0$ and this concludes the proof of (i).
The assertion (ii) can be proved by the same argument as the one used in \cite[Proposition 3.6]{LPV13}.
Combining (i) with (ii) derives the assertion (iii).
\end{proof}

Given a $\C^{\infty}$-K\"{a}hler differential 2-form $\omega\in\Omega_{X}^{2}(X)$, we will show that $\omega$ induces a natural sheaf morphism from $\mathcal{T}_{X}$ to  $\Omega^{1}_{X}$.
For any open subset $U\subset X$, define a morphism of $\C^{\infty}(U)$-modules
\begin{eqnarray}\label{omg-b-u}
  \omega^{\flat}(U): \mathcal{T}_{X}(U)&\longrightarrow&\Omega_{X}^{1}(U)\\
  v&\longmapsto& \iota(v)(\omega|_{U}). \nonumber
\end{eqnarray}
Assume that $U^{\prime}\subset X$ is another open subset such that $U^{\prime}\subset U$, then we have the corresponding morphism of $\C^{\infty}(U^{\prime})$-modules
\begin{equation}\label{omg-b-v}
\omega^{\flat}(U^{\prime}):\mathcal{T}_{X}(U^{\prime})
  \longrightarrow \Omega_{X}^{1}(U^{\prime}).
\end{equation}
Combining \eqref{omg-b-u}-\eqref{omg-b-v} and the restriction morphisms derives a commutative diagram:
\begin{equation*}\label{omg-u-v}
\vcenter{
\xymatrix@=1.3cm{
  \mathcal{T}_{X}(U) \ar[d]_{\rho_{UU^{\prime}}}
  \ar[r]^{\eqref{omg-b-u}} & \Omega^{1}_{X}(U) \ar[d]^{\bar{\rho}_{UU^{\prime}}} \\
  \mathcal{T}_{X}(U^{\prime}) \ar[r]^{\eqref{omg-b-v}} & \Omega^{1}_{X}(U^{\prime}).}
  }
\end{equation*}
This implies that \eqref{omg-b-u} defines a sheaf morphism from the tangent sheaf to the cotangent sheaf of $X$.
\begin{defn}\label{non-dege}
A $\C^{\infty}$-K\"{a}hler differential 2-form $\omega\in\Omega_{X}^{2}(X)$ is \emph{non-degenerate},
if the morphism of sheaves
\begin{equation}\label{omega-map}
  \omega^{\flat}:\mathcal{T}_{X} \longrightarrow \Omega^{1}_{X}
\end{equation}
is injective.
\end{defn}

Note that there exists a natural projection from $\Omega^{1}_{X}$ to the torsion-free cotangent sheaf $\Omega^{1}_{X,\mathrm{tf}}$.
Combining \eqref{omega-map} with the projection, we get a map of $\C^{\infty}_{X}$-modules
\begin{equation}\label{tf-omega-map}
  \omega^{\dag}:\mathcal{T}_{X} \longrightarrow \Omega^{1}_{X,\mathrm{tf}}.
\end{equation}


In particular, we have
\begin{prope}\label{t-inj}
The sheaf morphism \eqref{omega-map} is injective if and only if so does the morphism \eqref{tf-omega-map}.
\end{prope}
\begin{proof}
Consider the following commutative diagram of sheaf morphisms
\begin{equation*}
\vcenter{
\xymatrix@=1cm{
  \mathcal{T}_{X} \ar[dr]_{\omega^{\dag}} \ar[r]^{\omega^{\flat}}
                & \Omega^{1}_{X} \ar[d]^{\mathrm{pr}}  \\
                & \Omega^{1}_{X,\mathrm{tf}}. }
                }
\end{equation*}
A straightforward checking shows that the injectivity of $\omega^{\dag}$ implies that $\omega^{\flat}$ is injective.
We claim that the inverse is also true.
For any open subset $U$ in $X$, there exists a commutative diagram:
\begin{equation}\label{omg-reg}
\vcenter{
\xymatrix@=1.3cm{
  \mathcal{T}_{X}(U) \ar[d]_{\rho_{\mathrm{reg}}}
  \ar[r]^{\omega^{\flat}_{U}} & \Omega^{1}_{X}(U) \ar[d]^{\rho^{1}_{\mathrm{reg}}} \\
  \mathcal{T}_{X}(U_{\mathrm{reg}}) \ar[r]^{\omega^{\flat}_{\mathrm{reg}}} & \Omega^{1}_{X}(U_{\mathrm{reg}}).}
  }
\end{equation}
Suppose that $\omega^{\flat}$ is injective,
then $\omega^{\flat}_{U}$, $\omega^{\flat}_{\mathrm{reg}}$, and $\rho_{\mathrm{reg}}$ in \eqref{omg-reg} are injective.
Let $v$ be an element in $\mathcal{T}_{X}(U)$.
If $\omega^{\flat}_{U}(v)$ is a torsion element in $\Omega^{1}_{X}(U)$, i.e., $\rho^{1}_{\mathrm{reg}}\circ\omega^{\flat}_{U}(v)=0$,
then the commutativity of \eqref{omg-reg} implies $v=0$.
It follows that the image of $\omega^{\flat}_{U}$ is torsion-free.
Consequently, the morphism
$$
\omega^{\dag}_{U}:\mathcal{T}_{X}(U)\stackrel{\omega^{\flat}_{U}}
\longrightarrow\Omega^{1}_{X}(U)\stackrel{\mathrm{pr}}
\longrightarrow\Omega^{1}_{X,\mathrm{tf}}(U)
$$
is injective and so is the sheaf morphism $\omega^{\dag}$.
\end{proof}

We introduce the symplectic structure on a $\C^{\infty}$-stratified pseudomanifold as follows.

\begin{defn}\label{sym-spmd}
A \emph{symplectic stratified pseudomanifold} is a $\C^{\infty}$-stratified pseudomanifold $\mathcal{X}=(X,\mathcal{S}_{X},\C^{\infty}_{X})$ together with a closed non-degenerate $\C^{\infty}$-K\"{a}hler differential 2-form $\omega$, such that
the pullback $\underline{i}^\ast\omega\in\Omega^{2}(S)$ is an ordinary symplectic form on each stratum $S\in\MC_{X}$, where $\underline{i}=(i,i_{\sharp}):(S,\C^{\infty}_{S})\hookrightarrow (X,\C^{\infty}_{X})$ is the inclusion.
We call $\omega$ a \emph{symplectic form} on $\mathcal{X}$.
\end{defn}

We allow the stratum of dimension zero in the stratification $\mathcal{S}_{X}$, and consider it as a manifold with a \lq\lq trivial" symplectic form.
In general, if the non-degeneracy assumption is dropped, we call $\omega$ a \emph{presymplectic form} on $\mathcal{X}$.
In this case, a stratum $S\in\mathcal{S}_{X}$ is a \emph{presymplectic manifold} in the usual sense with the induced presymplectic form $\underline{i}^{\ast}\omega$.

In symplectic geometry, a fundamental result is the isotopy of symplectic forms attributed to Moser \cite{Mo65}.
One of the important consequences of the Moser theorem is the Darboux theorem which says that every symplectic form $\omega$ on a smooth manifold $M^{2n}$ is locally diffeomorphic to the standard symplectic form $\omega_{0}$ on $\mathbb{R}^{2n}$.
In the context of symplectic stratified pseudomanifolds, there is no obvious way to establish a Moser-type theorem since the singularities of stratified spaces.
Naturally, we have the following:
\begin{prob}
Does a Darboux-type theorem hold for symplectic stratified pseudomanifolds?
\end{prob}

We now consider the generalization of Liouville form in the setting of symplectic stratified pseudomanifolds.
Suppose $(\mathcal{X},\omega)$ is a symplectic stratified pseudomanifold of dimension $2n$.
Analogously, we say that the $\C^{\infty}$-K\"{a}hler differential form
\begin{equation}\label{liouville}
\frac{\omega^{n}}{n!}\in\Omega^{2n}(X)
\end{equation}
is the \emph{Liouville form} of $(\mathcal{X},\omega)$.
It is well-known that the Liouville form of a symplectic manifold $(M^{2n},\omega)$ is a \emph{volume form}, which means that as a global section of the vector bundle $\wedge^{2n}TM$ it is nonvanishing.
The singularity of $X$ leads to the fact that the sheaf $\Omega^{2n}_{X}$ can not be identified with the sheaf of sections of any vector bundle over $X$.
However, locally, the following proposition shows that the Liouville form also makes sense in the setting of symplectic stratified pseudomanifolds.
\begin{prop}
For any $x$ in $X$, the germ of the Liouville form \eqref{liouville} at $x$ is nonvanishing.
\end{prop}
\begin{proof}
For every $1\leq k\leq n$, from Definition \ref{sym-spmd}, we know that the $\C^{\infty}$-K\"{a}hler differential $2k$-form $\omega^{k}\in\Omega^{2k}(X)$ is not a torsion element in the sense of Definition \ref{tor-ele}.
If the assertion would not hold, then there exists a point $x$ such that the germ of $\omega^{n}$ at $x$ is zero.
This implies that $\omega^{n}$ vanishes on some open neighborhood $U$ of $x$ in $X$.
Consequently, we get $\omega^{n}|_{U_{\mathrm{reg}}}=0$ and this contradicts the fact that $\omega^{n}$ is not a torsion element.
\end{proof}

To mimic compact symplectic manifolds cohomologically, we introduce the following definition.
\begin{defn}\label{c-sym}
 Let $\mathcal{X}=(X,\mathcal{S}_{X},\C^{\infty}_{X})$ be a compact $\C^{\infty}$-stratified pseudomanifold with dimension $2n$.
 By a \emph{cohomologically symplectic (or c-symplectic)} structure on $X$, we mean a \v{C}ech cohomology class $[\omega]\in\check{\mathrm{H}}^{2}(M, \underline{\mathbb{R}})$ such that $[\omega^{n}]$ is nonzero in $\check{\mathrm{H}}^{2n}(M, \underline{\mathbb{R}})$.
\end{defn}
\subsection{Poisson bracket}\label{str-pos}

For a symplectic manifold $(M,\omega)$, it is useful to consider the so-called Poisson bracket on the algebra of smooth functions $\C^{\infty}(M)$.
Such bracket determines, and conversely is determined by, the symplectic form.
By contrast to the symplectic structure, the advantage of Poisson bracket is that it is a purely algebraic structure and can be generalized immediately to singular spaces such as algebraic varieties, (cf. \cite[Part I, \S 1.2]{LPV13}).
In complete analogy with the definition of Poisson bracket on a symplectic manifold,
we first consider the counterpart of Hamiltonian vector field in the stratified setting.

Let $(\mathcal{X},\omega)$ be a symplectic stratified pseudomanifold.
By definition, the $\C^{\infty}(X)$-linear map
$$
\omega^{\flat}:\mathcal{T}_{X}(X)\longrightarrow\Omega^{1}_{X}(X)
$$
is only injective.
So we can not define the derivation $D_{f}$ for an arbitrary smooth function $f$ on $X$ by the usual equation
\begin{equation}\label{h-equ}
\iota(D_{f})\omega=\mathrm{d}f.
\end{equation}
Denote by $\C^{\infty}(X, \omega)\subset\C^{\infty}(X)$ the subset of smooth functions such that $\mathrm{d}f\in\mathrm{im}\,(\omega^{\flat})$.
\begin{defn}\label{ham-v-f}
For any $f\in\C^{\infty}(X, \omega)$,
the unique global section $D_{f}\in\mathcal{T}_{X}(X)$ determined by the equation \eqref{h-equ}
is called the \emph{Hamiltonian derivation} associated to $f$.
\end{defn}

Next, we show that a symplectic form induces a natural bracket on $\C^{\infty}(X, \omega)$, and makes it into a Poisson algebra.
For any smooth functions $f$ and $g$ in $\C^{\infty}(X, \omega)$,
we can define a function
$$
\omega(D_{f},D_{g})=\iota(D_{f})\circ\iota(D_{g})\omega
=-\iota(D_{g})\circ\iota(D_{f})\omega=-\omega(D_{g},D_{f}).
$$
The smoothness of $\omega(D_{f},D_{g})$ follows from the following commutative diagram
\begin{equation*}
\xymatrix@=1.3cm{
  \C^{\infty}(X) \ar[d]_{\mathrm{d}} \ar[r]^{D_{f}} & \C^{\infty}(X) \\
  \Omega^{1}(X). \ar[ur]_{\iota(D_{f})}   }
\end{equation*}
This motivates the following definition.
\begin{defn}\label{possion}
The \emph{Poisson bracket} on a symplectic stratified pseudomanifold $(\mathcal{X},\omega)$ is a map defined by
\begin{equation*}
  \{-,-\}_{\omega}: \C^{\infty}(X, \omega)\times
  \C^{\infty}(X, \omega)\rightarrow
  \C^{\infty}(X, \omega),\,\,\,
  (f,g)\mapsto \omega(D_{f},D_{g}).
\end{equation*}
\end{defn}

Recall that a \emph{Poisson algebra} over $\mathbb{R}$ is a commutative associate $\mathbb{R}$-algebra $(\mathscr{A},\cdot)$ with a Lie bracket $\{-, -\}$ satisfying the \emph{Leibniz rule}:
  $$
  \{f\cdot g,h\}=f\cdot\{g,h\}+g\cdot\{f,h\},
  $$
for any elements $f$, $g$ and $h$ in $\mathscr{A}$.

If we endow $\C^{\infty}(X, \omega)$ with the usual multiplication of smooth functions, then it is obvious that $(\C^{\infty}(X, \omega), \cdot)$ becomes a commutative associate $\mathbb{R}$-algebra with unit $1$.
Moreover, combining with the bracket $\{-,-\}_{\omega}$ derives:
\begin{lem}\label{str-pos-alg}
The triple $\bigl(\C^{\infty}(X, \omega),\cdot,\{-,-\}_{\omega}\bigr)$ forms a Poisson algebra over $\mathbb{R}$.
\end{lem}
\begin{proof}
First, we show that $\{-,-\}_{\omega}$ satisfies the Jacobi identity.
According to the definition of Lie derivative \eqref{li-der} and Proposition \ref{cartan-fomla}, for any $f,g\in\C^{\infty}_{\omega}(X)$ we have
\begin{eqnarray*}
  \iota([D_{f},D_{g}])\omega
  &=& \mathcal{L}(D_{f})\circ\iota(D_{g})\omega- \iota(D_{g})\circ\mathcal{L}(D_{f})\omega\\
  &=&  \mathrm{d}\circ\iota(D_{f})\circ\iota(D_{g})\omega+
       \iota(D_{f})\circ\mathrm{d}\circ\iota(D_{g})\omega\\
  & &  -\iota(D_{g})\circ\mathrm{d}\circ\iota(D_{f})\omega-
       \iota(D_{g})\circ\iota(D_{f})\circ\mathrm{d}\omega\\
  &=&  \mathrm{d}\circ\iota(D_{f})\circ\iota(D_{g})\omega\\
  &=&  \mathrm{d}\bigl(\omega(D_{f},D_{g})\bigr)\\
  &=&  \mathrm{d}(\{f,g\}_{\omega}).
\end{eqnarray*}
This implies
$
D_{\{f,g\}_{\omega}}=[D_{f},D_{g}]
$
and therefore we get
\begin{eqnarray*}
  \{\{f,g\}_{\omega},h\}_{\omega}
  &=& \omega\bigl(D_{\{f,g\}_{\omega}},D_{h}\bigr) \\
  &=& \omega\bigl([D_{f},D_{g}],D_{h}\bigr) \\
  &=& [D_{f},D_{g}](h).
\end{eqnarray*}
Recall the definition of the Lie bracket of derivations.
Due to the previous equality, we obtain
\begin{eqnarray*}
  \{\{f,g\}_{\omega},h\}_{\omega}
   &=&  [D_{f},D_{g}](h)\\
   &=& D_{f}\cdot D_{g}(h)-D_{g}\cdot D_{f}(h) \\
   &=& D_{f}(\{g,h\}_{\omega})- D_{g}(\{f,h\}_{\omega})\\
   &=& \{f,\{g,h\}_{\omega}\}_{\omega}-\{g,\{f,h\}_{\omega}\}_{\omega},
\end{eqnarray*}
which is equivalent to the Jacobi identity
$$
\{\{f,g\}_{\omega},h\}_{\omega}+\{\{g,h\}_{\omega},f\}_{\omega}+\{\{h,f\}_{\omega},g\}_{\omega}=0.
$$
Via a straightforward manipulation, we can verify the Leibniz rule and this completes the proof.
\end{proof}

\begin{rem}
The notion of stratified symplectic space was first introduced by in \cite{SL91}.
For a differential-geometric study of singular symplectic quotients, it is necessary to consider the analogue of tangent bundles in the stratified setting.
From a viewpoint of  subcartesian space theory, Pflaum refined the notion of smooth structure in \cite{Pf01}.
Pflaum's version of a smooth structure, which we follow in this paper, is more differential-geometric.
In particular, if $X$ is a Whitney (A) smooth stratified space, then it allows a canonical construction of stratified tangent bundle $TX$.
Then a symplectic structure on $X$ is defined to be a smooth section
$$
\Lambda:X\longrightarrow TX\otimes TX
$$
such that, for any stratum $S$, the restriction
$\Lambda|_{S}:S\rightarrow TS\otimes TS$
yields a symplectic structure on $S$ (cf. \cite[Definition 2.6.1]{Pf01}).
Another refinement of stratified symplectic spaces was introduced by Somberg--L\^{e}--Van\v{z}ura \cite[Definition 3]{SLV13}.
We now briefly review it.
There a \emph{stratified symplectic form} on a stratified space $X$ is defined to be a collection of usual symplectic forms on strata (cf. \cite[Definition 6]{SLV13}).
In particular, if there exists a differential 2-form $\tilde{\omega}\in\Omega^{2}(X)$ in the sense of Mostow \cite[Section 2]{Mo79} such that the restriction of $\tilde{\omega}$ to each stratum $S$ coincides with $\omega_{S}$, then they say that the smooth structure $\C^{\infty}(X)$ is a weakly symplectic smooth structure (cf. \cite[Definition 6]{SLV13}).
Contrasting with the notions of symplectic structures above, the definition of symplectic stratified pseudomanifold (Definition \ref{sym-spmd}) is sheaf-theoretic.
This enable us to consider the symplectic manifolds and the symplectic stratified psuedomanifolds in a uniform framework of local $\C^{\infty}$-ringed spaces, and apply various techniques in $\C^{\infty}$-algebraic geometry to studying the symplectic geometry of smooth stratified pseudomanifolds.
\end{rem}

\subsection{K\"{a}hler stratified pseudomanifolds}

Consider a complex vector space $\mathbb{C}^{m}$.
Denote $z_{i}=x_{i}+\sqrt{-1}\cdot y_{i}$ for each $1\leq i\leq m$.
Then we can view $\mathbb{C}^{m}$ as the $2m$-dimensional real vector space $\mathbb{R}^{2m}$ and the topological structure of $\mathbb{C}^{m}$ is identical to that of $\mathbb{R}^{2m}$.
Suppose that $\mathcal{X}=(X,\mathcal{S}_{X},\C^{\infty}_{X})$ is a smooth stratified pseudomanifold, we now introduce the complex analytic structure on $\mathcal{X}$.
\begin{defn}\label{c-s-psd}
Given a point $x$ in $X$, by a \emph{holomorphic singular chart} around $x$ we mean a pair $(U,\phi)$ such that the following conditions hold:
\begin{itemize}
  \item [(i)] $U\subset X$ is an open neighborhood of $x$;
  \item [(ii)] $\phi:U\rightarrow
      \mathrm{im}\,(\phi)\subset\mathbb{C}^{m}$
      is a homeomorphism such that the image $\mathrm{im}\,(\phi)$ is a locally closed subset of $\mathbb{C}^{m}$;
  \item [(iii)]
      $\phi(U\cap S)$ is a complex submanifold in $\mathbb{C}^{m}$ for each stratum $S\in\mathcal{S}_{X}$.
\end{itemize}
\end{defn}

Given two holomorphic singular charts $\phi_{U}:U\rightarrow\mathbb{C}^{m_{1}}$ and
$\phi_{V}:V\rightarrow\mathbb{C}^{m_{2}}$ satisfying $U\cap V\neq\emptyset$.
Akin to the compatibility for $\C^{\infty}$-singular charts in Definition \ref{ssp}, we say that $(U,\phi_{U})$ and $(V,\phi_{V})$ are \emph{compatible}, if for every $x\in U\cap V$ there exist an open neighborhood $\mathcal{U}_{1}$ of $\phi_{U}(x)$ in $\mathbb{C}^{m_{1}}$,
an open neighborhood $\mathcal{U}_{2}$ of $\phi_{V}(x)$ in $\mathbb{C}^{m_{2}}$,
and \emph{holomorphic} mappings $H_{1}:\mathcal{U}_{1}\rightarrow\mathbb{C}^{m_{2}}$ and
$H_{2}:\mathcal{U}_{2}\rightarrow\mathbb{C}^{m_{1}}$ such that $H_{1}=\phi_{V}\circ\phi^{-1}_{U}$ on $\phi_{U}(U\cap V)\cap\mathcal{U}_{1}$ and $H_{2}=\phi_{U}\circ\phi^{-1}_{V}$ on $\phi_{V}(U\cap V)\cap\mathcal{U}_{2}$.

\begin{defn}\label{hol-stru}
By a \emph{holomorphic structure} on a smooth stratified pseudomanifold $\mathcal{X}$ we mean a collection of holomorphic singular charts
$$
\mathfrak{A}=
\{(U,\phi_{U}),(V,\phi_{V}),(W,\phi_{W}),\cdots\}
$$
on $X$ such that:
\begin{itemize}
  \item [(i)] $\{U,V,W,\cdots\}$ constitutes an open covering of $X$;
  \item [(ii)] any two holomorphic singular charts in $\mathfrak{A}$ are compatible with each other;
  \item [(iii)] $\mathfrak{A}$ is maximal: if a holomorphic singular chart $(U^{\prime},\phi^{\prime})$ is compatible with all elements in $\mathfrak{A}$, then $(U^{\prime},\phi^{\prime})\in\mathfrak{A}$.
\end{itemize}
A \emph{complex stratified pseudomanifold} is a smooth stratified pseudomanifold $\mathcal{X}$ endowed with a holomorphic structure $\mathfrak{A}$.
\end{defn}

From definition, all strata of a complex stratified pseudomanifold are complex manifolds.
\begin{defn}\label{hol-fun}
A function $f:X\rightarrow\mathbb{C}$ on a complex stratified pseudomanifold $\mathcal{X}$ is \emph{holomorphic},
if for any holomorphic singular chart $\phi:U\rightarrow\mathbb{C}^{m}$
of a holomorphic structure $\mathfrak{A}$ defining it, there exists a holomorphic function $F:\mathbb{C}^{m}\rightarrow\mathbb{C}$ such that $f|_{U}=F\circ\phi$.
\end{defn}

Similar to complex manifolds, we denote by $\CO_{X}$ the \emph{sheaf of holomorphic functions} on $X$, i.e., for any open subset $U\subset X$ one has
$$
\CO_{X}(U)=\Gamma(U,\CO_{X})
=\bigl\{f:U\rightarrow\mathbb{C}\,|\,f\,\,\,\mathrm{is}\,\,\,\mathrm{holomorphic}\bigr\}.
$$
In general, we write a complex stratified pseudomanifold as a triple $\mathcal{X}=(X,\MC_{X},\CO_{X})$ and call $\CO_{X}$ the \emph{structure sheaf} of $\mathcal{X}$.
By definition, given a holomorphic singular chart $\phi:U\rightarrow \mathbb{C}^{m}$, the morphism
$\phi^{\ast}:\CO(\mathbb{C}^{m})\rightarrow\CO_{X}(U)$ is \emph{surjective} and therefore the algebra of holomorphic functions over $U$ is canonically isomorphic to the algebra $\CO(\mathbb{C}^{m})\big/\mathcal{I}$, where $\mathcal{I}$ is the ideal of holomorphic functions on $\mathbb{C}^{m}$ which are vanishing on $\phi(U)$.

\begin{defn}\label{com-str-mp}
Let $\mathcal{X}=(X,\mathcal{S}_{X},\CO_{X})$
and
$\mathcal{Y}=(Y,\mathcal{S}_{Y},\CO_{Y})$
be two complex stratified pseudomanifolds.
A smooth map $f:X\rightarrow Y$ is a \emph{holomorphic map} if it satisfies:
\begin{itemize}
  \item [(i)] $f$ maps a stratum of $\MC_{X}$ to a stratum of $\MC_{Y}$;
  \item [(ii)] for each open subset $V$ in $Y$ and each $g\in\CO_{Y}(V)$ there holds $f^{\ast}(g)\in\CO_{X}(f^{-1}(V))$.
\end{itemize}
\end{defn}
If we view a complex stratified pseudomanifold as a reduced local ringed space over $\mathbb{C}$, then a holomorphic map of complex stratified pseudomanifolds $f:\mathcal{X}\rightarrow\mathcal{Y}$ actually defines a morphism of local ringed spaces
$\underline{f}=(f,f_{\sharp}):(X,\CO_{X})\rightarrow(Y,\CO_{Y})$.
Moreover, complex stratified pseufomanifolds and holomorphic maps constitute a category containing the category of complex manifolds as a full subcategory.

\begin{ex}[Complex analytic varieties]\label{com-ana-variety}
Here is a basic example of complex stratified pseudomanifold from complex analytic geometry.
Assume that $V$ is an analytic variety of a complex manifold $M$.
By Example \ref{c-a-v}, we know that $V$ is a smooth stratified pseudomanifold with respect to a canonical stratification $\mathcal{S}_{V}$.
Consider the holomorphic atlas of $M$:
$$
\mathfrak{A}_{M}=
\bigl\{\phi_{\lambda}:U_{\lambda}\rightarrow
\mathbb{C}^{n}\,\big|\,
\lambda\in\Lambda\bigr\},
$$
where $n=\mathrm{dim}_{\mathbb{C}}\,(M)$.
From definition, the collection
$$
\mathfrak{A}_{V}:=
\bigl\{\phi_{\lambda}|_{V\cap U_{\lambda}}:V\cap U_{\lambda}\rightarrow
\mathbb{C}^{n}\,\big|\,
\lambda\in\Lambda\bigr\}
$$
yields a holomorphic structure of $V$ in the sense of Definition \ref{hol-stru}, and hence $V$ becomes a complex stratified pseudomanifold.
In particular, it follows that both complex affine varieties and projective varieties are complex stratified pseudomanifolds.
\end{ex}

\begin{ex}[Complex spaces]\label{com-spa}
Let $(X,\CO_{X})$ be a reduced complex space.
Then there exists a canonical stratification by singularities, denoted by $\mathcal{S}^{\mathrm{sing}}_{X}$, which has complex manifolds as strata (cf. \cite[Proposition 5.6]{Dem12}).
For every $x$ in $X$, let
\begin{equation}\label{chart-x}
\phi_{x}: U_{x}\longrightarrow D\subset \mathbb{C}^{K}
\end{equation}
be the local embedding of $X$ into an open domain $D$ of $\mathbb{C}^{K}$.
Here $U_{x}$ is an open neighborhood around $x$ and $K$ depends on $x$.
It follows from the definition that \eqref{chart-x} gives rise to a holomorphic singular chart on $X$.
Moreover, we get a collection of holomorphic singular charts:
$$
\mathfrak{A}_{X}=\{(U_{x},\phi_{x})\,|\, x\in X\}.
$$
We now have to verify the compatibility of each pair of elements in $\mathfrak{A}_{X}$.
Given another point $y\in X$,
we denote the holomorphic singular chart around $y$ by
\begin{equation*}\label{chart-y}
\phi_{y}:U_{y}\longrightarrow D^\prime \subset\mathbb{C}^{L}.
\end{equation*}
Set $W=U_{x}\cap U_{y}$.
Then $V:=\phi_{y}(W)$ is an open subset in the complex analytic set $\phi_{y}(U_{y})$ and hence we have
\begin{equation}\label{V-function}
\mathcal{O}_{X}|_{W}\cong(\phi_{y}|_{W})^{\ast}\mathcal{O}_{V}
=(\phi_{y}|_{W})^{\ast}\mathcal{O}_{D^{\prime}}.
\end{equation}
For each $1\leq i\leq K$, let $\pi_{i}:\mathbb{C}^{K}\rightarrow\mathbb{C}$ be the projection onto the $i$-th coordinate.
Observe that $\pi_{i}\circ\phi_{x}:W\rightarrow\mathbb{C}$ is a holomorphic function on $W$.
Due to \eqref{V-function}, for any $z\in W$, there exist open neighborhoods $P_{i}\subset W$ of $z$ and $Q_{i}\subset D^{\prime}$ of $\phi_{y}(z)$ such that $\phi_{y}(P_{i})\subset Q_{i}$ and
$$
(\pi_{i}\circ\phi_{x})|_{P_{i}}=(\phi_{y}|_{P_{i}})^{\ast}(h_{i})=h_{i}\circ(\phi_{y}|_{P_{i}})
$$
for some holomorphic function $h_{i}\in\mathcal{O}(Q_{i})$.
Set $P=\cap_{i=1}^{K}P_{i}$, $Q=\cap_{i=1}^{K}Q_{i}$ and $H=(h_{1},\cdots,h_{K})$.
Then we get a commutative diagram:
\begin{equation*}\label{l-to-k}
\vcenter{
\xymatrix@=1.3cm{
  P \ar[d]_{\phi_{y}} \ar[r]^{\phi_{x}} &\mathbb{C}^{K}      \\
  Q \ar[ur]_{H}}
  }
\end{equation*}
As a result, we have
$H=\phi_{x}\circ\phi^{-1}_{y}$ on $\phi_{y}(P)$.
Using the same argument, we can also verify that $\phi_{y}\circ\phi^{-1}_{x}$ has a local holomorphic extending.
From definition, $\mathfrak{A}_{X}$ determines a holomorphic structure on $X$.
As a result, together with the stratification $\mathcal{S}^{\mathrm{sing}}_{X}$, the complex space $(X,\CO_{X})$ admits a natural structure of complex stratified pseudomanifold.
\end{ex}

\begin{defn}\label{k-s-p}
A \emph{K\"{a}hler stratified pseudomanifold} is a complex stratified pseudomanifold $(X,\mathcal{S}_{X},\CO_{X})$ equipped with a symplectic form $\omega$, such that
the pullback $\underline{i}^\ast\omega$ is a usual K\"{a}hler structure on each stratum $S\in\MC_{X}$, where $\underline{i}=(i,i_{\sharp}):(S,\mathcal{O}_{S})\hookrightarrow (X,\CO_{X})$ is the inclusion.
\end{defn}

\begin{rem}
In the literature, there are different notions of K\"{a}hler structures on stratified spaces.
As a continuation  of \cite{Va89}, Heinzner--Huckleberry--Loose \cite{HHL94} introduced the notion of stratified K\"{a}hlerian space which is a K\"{a}hler space equipped with a complex stratification.
Based on Sjamaar--Lerman's stratified symplectic spaces, Huebschmann \cite[Section 2]{Hu04} proposed another definition of K\"{a}hler structure in the stratified setting.
More precisely, a stratified K\"{a}hler space in the sense of Huebschmann is a stratified symplectic space together with a stratified K\"{a}hler polarization.
\end{rem}
\section{Singular K\"{a}hler spaces}\label{sig-kah}
In this section, we first show that every singular K\"{a}hler space in the sense of Moishezon \cite{Mo75} admits a natural structure of K\"{a}hler stratified pseudomanifold, and then we consider the K\"{a}hler class corresponding to a singular K\"{a}hler quotient.
Throughout of this section, we assume that $(X,\CO_{X})$ is a \emph{reduced normal} complex space.

\subsection{K\"{a}hler classes and K\"{a}hler forms}\label{k-c-c}
First we give a brief review on some basic notations on K\"{a}hler spaces.
By a \emph{real-valued pluriharmonic} function on $X$, we mean a smooth function which locally is the real part of some holomorphic function, and we define the sheaf $\mathscr{PH}_{X,\mathbb{R}}$ of real-valued pluriharmonic functions.
We say that a continuous (resp. smooth) function $f: X\rightarrow\mathbb{R}$ is a continuous (resp. smooth) \emph{plurisubharmonic (p.s.h.)} function on $X$, if locally it is the restriction of a continuous (resp. smooth)  p.s.h. function on an open subset in $\mathbb{C}^{N}$ for a local embedding of $X$ into $\mathbb{C}^{N}$ (cf. \cite[Section 5]{FN80}).
In particular, we call $f$ a \emph{strictly p.s.h. function} (in the sense of perturbations), if for each smooth function $h$ and each relatively compact open subset $U$ there exists $\epsilon>0$ such that $f+t\cdot h$ is p.s.h. on $U$ for all $|t|<\epsilon$ (cf. \cite[Definition 55]{HS20}).
We denote by sheaves $\mathscr{P}^{0}_{X}$ (resp. $\mathscr{P}^{\infty}_{X}$) of continuous (resp. smooth) p.s.h.  functions, and $\mathscr{SP}^{0}_{X}$ (resp. $\mathscr{SP}^{\infty}_{X}$) of continuous (resp. smooth) strictly p.s.h. functions.

\begin{defn}(cf. \cite{Mo75})\label{kahler space}
A \emph{K\"{a}hler metric} on $X$ is defined to be an open covering $\{U_{\alpha}\}_{\alpha\in\Lambda}$ of $X$ together with a collection of strictly p.s.h. $\C^{\infty}$-functions $\{\varphi_{\alpha}:U_{\alpha}\rightarrow\mathbb{R}\}_{\alpha\in\Lambda}$ such that $(\varphi_{\alpha}-\varphi_{\beta})|_{U_{\alpha}\cap U_{\beta}}$ is pluriharmonic when  $U_{\alpha}\cap U_{\beta}\neq\emptyset$.
A \emph{K\"{a}hler space} is a complex space equipped with a K\"{a}hler metric.
\end{defn}

In the literature,  a K\"{a}hler metric on $X$ is also defined to be a system of \emph{continuous} strongly p.s.h. functions $\varphi_{\alpha}\in\mathscr{SP}^{0}_{X}(U_{\alpha})$ corresponding to an open covering $\{U_{\alpha}\}_{\alpha\in\Lambda}$ such that $(\varphi_{\alpha}-\varphi_{\beta})|_{U_{\alpha}\cap U_{\beta}}\in\mathscr{PH}_{X, \mathbb{R}}(U_{\alpha}\cap U_{\beta})$ when $U_{\alpha}\cap U_{\beta}\neq\emptyset$.
To avoid confusion, we call such a collection $\{(U_{\alpha}, \varphi_{\alpha})\}_{\alpha\in\Lambda}$ a K\"{a}hler metric with \emph{continuous} local potentials.

Set $\C^{\infty}_{X,\mathbb{R}}$ the sheaf of real-valued smooth functions and $\mathscr{K}^{1}_{X,\mathbb{R}}=\C^{\infty}_{X,\mathbb{R}}/\mathscr{PH}_{X,\mathbb{R}}$.
Equivalently, by definition, a K\"{a}hler metric $\kappa:=\{(U_{\alpha}, \varphi_{\alpha})\}_{\alpha\in\Lambda}$ on $X$ corresponds to a class in the 0-th \v{C}ech cohomology $\check{\mathrm{H}}^{0}(X, \mathscr{K}^{1}_{X,\mathbb{R}})$.
In particular, two collections $\{(U_{\alpha}, \varphi_{\alpha})\}_{\alpha\in\Lambda}$ and $\{(V_{i}, \phi_{i})\}_{i\in\Lambda^{\prime}}$ define the same K\"{a}hler metric on $X$ if and only if
$$
(\varphi_{\alpha}-\phi_{i})|_{U_{\alpha}\cap V_{i}}\in \mathscr{PH}_{X,\mathbb{R}}(U_{\alpha}\cap V_{i}),
$$
for all $U_{\alpha}\cap V_{i}\neq\emptyset$.
For a given K\"{a}hler metric $\kappa$ on $X$, we can associate to $\kappa$ a cohomology class called the \emph{K\"{a}hler class} as follows (cf. \cite[\S 3.1]{GK20}).
Consider the natural short exact sequence of sheaves on $X$:
\begin{equation}\label{s-e-ph}
\xymatrix@C=0.5cm{
  0 \ar[r] & \mathscr{PH}_{X,\mathbb{R}} \ar[rr]^{} && \C^{\infty}_{X,\mathbb{R}}
   \ar[rr]^{} && \mathscr{K}^{1}_{X,\mathbb{R}} \ar[r] & 0. }
\end{equation}
From \eqref{s-e-ph}, we obtain a canonical connecting homomorphism
\begin{equation}\label{delta1}
\delta^{0}: \check{\mathrm{H}}^{0}(X, \mathscr{K}^{1}_{X,\mathbb{R}})\longrightarrow \check{\mathrm{H}}^{1}(X, \mathscr{PH}_{X,\mathbb{R}}).
\end{equation}
Let $\underline{\mathbb{R}}$ be the sheaf of locally constant functions on $X$ and $\underline{\mathbb{R}}\hookrightarrow\mathcal{O}_{X}$ the embedding via multiplication by $\sqrt{-1}$.
Then we get a short exact sequence
\begin{equation}\label{re-part}
\xymatrix@C=0.5cm{
  0 \ar[r] & \underline{\mathbb{R}} \ar[rr]^{\sqrt{-1}\cdot} && \mathcal{O}_{X}
   \ar[rr]^{\mathrm{Re}\quad} && \mathscr{PH}_{X,\mathbb{R}} \ar[r] & 0,}
\end{equation}
where $\mathrm{Re}$ is the real part map.
By \eqref{re-part}, we deduce another connecting homomorphism
\begin{equation}\label{delta2}
\delta^{1}: \check{\mathrm{H}}^{1}(X, \mathscr{PH}_{X,\mathbb{R}})\longrightarrow \check{\mathrm{H}}^{2}(X, \underline{\mathbb{R}}).
\end{equation}
Composing \eqref{delta1} with \eqref{delta2} gives rise to a canonical morphism:
\begin{equation}\label{c}
c_{1}: \check{\mathrm{H}}^{0}(X, \mathscr{K}^{1}_{X,\mathbb{R}})\longrightarrow \check{\mathrm{H}}^{2}(X, \underline{\mathbb{R}}).
\end{equation}
The cohomology class $c_{1}(\kappa)\in \check{\mathrm{H}}^{2}(X, \underline{\mathbb{R}})$ is called the \emph{K\"{a}hler class} of $\kappa$.

Similar to K\"{a}hler manifolds, a K\"{a}hler metric on a singular complex space yields a closed $\C^{\infty}$-K\"{a}hler differential 2-form called the K\"{a}hler form.
Consider the case where $X$ is an analytic set of a domain $D$ in $\mathbb{C}^{N}$.
By Example \ref{com-spa}, $X$ is a complex stratified pseudomanifold with respect to the canonical stratification by singularities.
Set $\C^{\infty}_{\mathbb{C}}(X)$ the algebra of $\mathbb{C}$-valued smooth functions and $\Omega^{k}_{X}(X;\mathbb{C})$ the associated $\C^{\infty}$-K\"{a}hler differentials $k$-forms.
Let $\underline{i}=(i,i_{\sharp}):(X,\mathcal{O}_{X})\hookrightarrow(D,\mathcal{O}_{D})$ be the holomorphic inclusion.
From definition, there is a canonical surjective morphism $\underline{i}^\ast:\Omega^{k}_{D}(D;\mathbb{C})\rightarrow\Omega^{k}_{X}(X;\mathbb{C})$.
Note that there exists a natural decomposition
$$
\Omega^{k}_{D}(D;\mathbb{C})=\bigoplus_{p+q=k}\Omega^{p,q}_{D}(D;\mathbb{C}).
$$
Set $\Omega^{k,l}_{X}(X;\mathbb{C})$ the image of $\Omega^{k,l}_{D}(D;\mathbb{C})$ under $\underline{i}^\ast$.
Then we have
$$
\Omega^{k}_{X}(X;\mathbb{C})=\bigoplus_{p+q=k}\Omega^{p,q}_{X}(X;\mathbb{C}).
$$
The usual differential
$\mathrm{d}=\partial+\bar{\partial}:
\Omega^{k}_{D}(D;\mathbb{C})\rightarrow\Omega^{k+1}_{D}(D;\mathbb{C})$
naturally induces the one acting on $\C^{\infty}$-K\"{a}hler differential forms
$\mathrm{d}=\partial+\bar{\partial}:
\Omega^{k}_{X}(X;\mathbb{C})\rightarrow\Omega^{k+1}_{X}(X;\mathbb{C})$
which satisfies the identities $\mathrm{d}^{2}=\partial^{2}=\bar{\partial}^{2}=\partial\bar{\partial}+\bar{\partial}\partial=0$.
In general case, let $\Omega^{p,q}_{X}$ be the sheaf of germs of $(p,q)$-type $\C^{\infty}$-K\"{a}hler differential forms on the complex space $X$.
Then we get the resulting $\bar{\partial}$- and $\partial$-complexes of sheaves, which are \emph{not} exact in general for the presence of singularities.

Let $\kappa=\{(U_{\alpha}, \varphi_{\alpha})\}_{\alpha\in\Lambda}$ be a K\"{a}hler metric on $X$.
From definition, for any $U_{\alpha}\cap U_{\beta}\neq\emptyset$ we have $\mathrm{d}\mathrm{d}^{c}\varphi_{\alpha}=\mathrm{d}\mathrm{d}^{c}\varphi_{\beta}$ on $U_{\alpha}\cap U_{\beta}$, where $\mathrm{d}^{c}=\sqrt{-1}(\bar{\partial}-\partial)$.
This implies that $\{(\mathrm{d}\mathrm{d}^{c}\varphi_{\alpha},U_{\alpha})\}_{\alpha\in\Lambda}$ defines a globally closed real $(1,1)$-type $\C^{\infty}$-K\"{a}hler differential form denoted by $\omega$ which is called the \emph{K\"{a}hler form} of the metric.
Moreover, $\omega$ represents a cohomology class in the $\mathcal{C}^{\infty}$-algebraic de Rham cohomology $\mathrm{H}^{2}_{\mathrm{DR}}(X)$.
Especially, if $X$ is smooth and $h$ is a K\"{a}hler metric with the K\"{a}hler form $\omega$.
The local K\"{a}hler potentials of $\omega$ yields a unique element $\kappa\in \check{\mathrm{H}}^{0}(X, \mathscr{K}^{1}_{X,\mathbb{R}})$.
In this case, we can identify the K\"{a}hler class $c_{1}(\kappa)$
with the de Rham cohomology class represented by $\omega$ via the canonical isomorphism $\check{\mathrm{H}}^{2}(X, \underline{\mathbb{R}})\cong \mathrm{H}^{2}_{\mathrm{DR}}(X)$.

\begin{thm}\label{com-spa-wk}
Let $(X,\CO_{X})$ be a complex space with a K\"{a}hler metric $\kappa=\{(U_{\alpha}, \varphi_{\alpha})\}_{\alpha\in\Lambda}$.
Then the K\"{a}hler form $\omega=\{(\mathrm{d}\mathrm{d}^{c}\varphi_{\alpha},U_{\alpha})\}_{\alpha\in\Lambda}$ makes $X$ into a K\"{a}hler stratified pseudomanifold with respect to the stratification by singularities.
\end{thm}
\begin{proof}
Note that $(X,\mathcal{S}^{\mathrm{sing}}_{X},\CO_{X})$ is a complex stratified pseudomanifold.
Let $S$ be an arbitrary stratum of $X$.
We show that the pullback $\underline{\imath}^{\ast}\omega$ defines a usual K\"{a}hler structure on $S$,
where $\underline{\imath}=(\imath,\imath_{\sharp}):(S,\CO_{S})\hookrightarrow (X,\CO_{X})$ is the inclusion map.
Without loss of generality, we may assume that $X$ is an analytic subset of a domain $D$ of $\mathbb{C}^{n}$ and the K\"{a}hler metric is given by a strictly plurisubharmonic $\mathcal{C}^{\infty}$-function $f:X\rightarrow\mathbb{R}$.
Consequently, we get a commutative diagram of holomorphic maps:
\begin{equation}\label{SXD}
\vcenter{
\xymatrix@C=1.5cm{
  S \ar[d]_{\underline{\imath}} \ar[dr]^{\underline{\jmath}}        \\
  X \ar[r]_{\underline{\psi}}  & D              }
  }
\end{equation}
From definition, there exists a strictly plurisubharmonic $\mathcal{C}^{\infty}$-function $f:D\rightarrow\mathbb{R}$ such that $\varphi=\underline{\psi}^{\ast}f$.
Observe that $\Omega:=\mathrm{d}\mathrm{d}^{c}f$ defines a K\"{a}hler metric on $D$.
Since $S\subset D$ is a complex submanifold the pullback of the metric $\underline{\jmath}^{\ast}\Omega$ gives rise to a usual K\"{a}hler metric on $S$.
The commutativity of \eqref{SXD} implies $\underline{\imath}^{\ast}\omega=\underline{\imath}^{\ast}\circ\underline{\psi}^{\ast}(\Omega)
=\underline{\jmath}^{\ast}\Omega$
and therefore $\underline{\imath}^{\ast}\omega$ defines a usual K\"{a}hler structure on $S$.

To finish the proof, it remains to verify that $\omega$ is non-degenerate.
For all open subset $U$ of $X$, let $U_{\mathrm{reg}}$ be the intersection of $U$ with the regular stratum $X_{\mathrm{reg}}$ and $\omega_{\mathrm{reg}}$ the restriction of $\omega$ to $U_{\mathrm{reg}}$.
Then $\omega_{\mathrm{reg}}$ defines a usual K\"{a}hler metric on $U_{\mathrm{reg}}$.
Moreover, we get a commutative diagram
\begin{equation}\label{omg-reg-1}
\vcenter{
\xymatrix@C=1.3cm{
  \mathcal{T}_{X}(U) \ar[d]_{\rho_{\mathrm{reg}}}
  \ar[r]^{\omega^{\flat}_{U}} & \Omega^{1}_{X}(U) \ar[d]^{\rho^{1}_{\mathrm{reg}}} \\
  \mathcal{T}_{X}(U_{\mathrm{reg}}) \ar[r]^{\omega_{\mathrm{reg}}^{\flat}}_{\simeq} & \Omega^{1}_{X}(U_{\mathrm{reg}}).}
  }
\end{equation}
As $\rho_{\mathrm{reg}}$ is injective, so is the map $\omega^{\flat}_{U}$ in \eqref{omg-reg-1} by the commutativity.
This implies that the sheaf morphism $\omega^{\flat}:\mathcal{T}_{X}\rightarrow\Omega^{1}_{X}$ is injective, i.e., $\omega$ is non-degenerate.
\end{proof}

\begin{rem}
For a given complex orbifold $X$, there are two natural ways to extend the K\"{a}hler structure on it.
Akin to Hermitian manifold, one can endow a complex orbifold with an orbifold Hermitian metric.
The complex orbifold $X$ is said to be K\"{a}hler, if it admits an orbifold Hermitian metric such that the associated orbifold K\"{a}hler form is closed.
From complex analytic geometry point of view, the complex orbifold $X$ can be considered as a complex analytic space and thus we can define $X$ to be K\"{a}hler if it is a K\"{a}hler space in the sense of Definition \ref{kahler space}.
A $\C^{\infty}$-K\"{a}hler differential $p$-form on $X$ always gives rise to an orbifold differential $p$-form.
However, in general, an orbifold differential $p$-form on $X$ can not descend to a $\C^{\infty}$-K\"{a}hler differential $p$-form.
Particularly, the two extensions of K\"{a}hler structures on a compact complex orbifold are equivalent, see \cite{Wu23}.
\end{rem}
\subsection{K\"{a}hler quotients}
Consider a K\"{a}hler Hamiltonian $G$-manifold $(M,ds^{2},G,\mu)$.
Assume that the moment map $\mu$ is proper.
It has been shown that the K\"{a}hler quotient $M_{0}\cong M^{ss}/\!\!/G^{\mathbb{C}}$ is a compact reduced normal complex analytic space (cf. \cite{Sj95, HS20}).
Furthermore, Heinzner--Stratmann showed that the original K\"{a}hler metric on $M$ induces a natural K\"{a}hler metric with continuous local potentials on $M_{0}$.

To be more specific, by a holomorphic function on $Z$ we mean the restriction of a holomorphic function on $M$.
Let $\mathcal{O}_{Z}$ be the sheaf of holomorphic functions on $Z$ and $\mathcal{O}_{M_{0}}=(\pi_{\ast}\mathcal{O}_{Z})^{G}$.
On account of \cite[Theorem 2]{HS20}, for any $x$ in $Z=\mu^{-1}(0)$ there exist a $G$-invariant open neighborhood $U$ of $x$ together with a $G$-invariant smooth strictly p.s.h.  function $\varphi: U\rightarrow\mathbb{R}$ satisfying the following conditions:
\begin{itemize}
  \item [(i)] the analytic Hilbert quotient in the sense of Heinzner \cite{Hei91} $\pi_{\mathrm{an}}:U\rightarrow U/\!\!/G$ is well-defined and $(U/\!\!/G, \mathcal{O}_{U/\!\!/G})$ is a reduced normal complex space;
  \item [(ii)] the inclusion $\imath:Z_{U}:=Z\cap U\hookrightarrow U$ induces a biholomorphic map
      $$\chi:(Z_{U}/G, \mathcal{O}_{Z_{U}/G})\longrightarrow(U/\!\!/G, \mathcal{O}_{U/\!\!/G})$$
      such that the following diagram is commutative;
      \begin{equation}\label{zu-u}
      \vcenter{
      \xymatrix{
        Z_{U} \ar[d]_{\pi} \ar[r]^{\imath} & U \ar[d]^{\pi_{\mathrm{an}}} \\
        Z_{U}/G \ar[r]^{\chi} & U/\!\!/G   }
        }
      \end{equation}
  \item [(iii)] $\omega|_{U}=\mathrm{d}\mathrm{d}^{c}\varphi$ and
      $(\mu|_{U})^{v}=-\iota(\tilde{v})\mathrm{d}^{c}\varphi$ for any $v\in\mathfrak{g}$, where $\tilde{v}$ is the fundamental vector field on $U$ generated by $v$.
\end{itemize}
The pair $(U, \varphi)$ is called a local slice model around $x\in Z$.
As $\mu$ is proper, the level set $Z$ is compact and therefore we get a finite family of $G$-invariant open subsets $\mathscr{U}:=\{U_{1}, \cdots, U_{l}\}$ together with a collection of $G$-invariant smooth strictly p.s.h.  function $\{\varphi_{i}: U_{i}\rightarrow\mathbb{R}\}^{l}_{i=1}$ such that $Z\subset\bigcup^{l}_{i=1}U_{i}$.
Set $\mathscr{U}_{Z}:=\{Z\cap U_{1}, \cdots, Z\cap U_{l}\}$ and $\tilde{\mathscr{U}}:=\{\tilde{U}_{1}, \cdots, \tilde{U}_{l}\}$ where $\tilde{U}_{i}:=\pi(Z\cap U_{i})\cong U_{i}/\!\!/G$.
Then $\tilde{\mathscr{U}}$ forms an open covering of $M_{0}$ and every function $\varphi_{i}|_{Z\cap U_{i}}$ is $G$-invariant.
Furthermore, $\varphi_{i}|_{Z\cap U_{i}}$ descends to a \emph{continuous} strictly p.s.h. function $\tilde{\varphi}_{i}$ on the analytic Hilbert quotient $U_{i}/\!\!/G\cong\tilde{U}_{i}$ satisfying the condition
$$
(\tilde{\varphi}_{i}-\tilde{\varphi}_{j})|_{\tilde{U}_{i}\cap \tilde{U}_{j}}\in
\mathscr{PH}_{M_{0}, \mathbb{R}}(\tilde{U}_{i}\cap \tilde{U}_{j})
$$
on each $\tilde{U}_{i}\cap \tilde{U}_{j}\neq\emptyset$.
Put $\C^{0}_{M_{0}, \mathbb{R}}$ the sheaf of real-valued continuous functions and denote by $\mathscr{F}_{M_{0}, \mathbb{R}}$ the quotient sheaf $\C^{0}_{M_{0}, \mathbb{R}}/\mathscr{PH}_{M_{0}, \mathbb{R}}$.
It follows that the collection $\kappa=\{(\tilde{U}_{i}, \tilde{\varphi}_{i})\}^{l}_{i=1}$ gives rise to an element in $\check{\mathrm{H}}^{0}(M_{0}, \mathscr{F}_{M_{0}, \mathbb{R}})$.

Observe that $\varphi_{i}|_{Z\cap U_{i}}$ also can be considered as a function on the quotient $(Z\cap U_{i})/G$ denoted by $\hat{\varphi}_{i}$.
By the definition of smooth structure on the symplectic quotient (see Example \ref{sym-quot}), we get $\hat{\varphi}_{i}\in\mathcal{C}^{\infty}((Z\cap U_{i})/G)$.
Although the quotient space $(Z\cap U_{i})/G$ can be identified with the analytic Hilbert quotient $U_{i}/\!\!/G\cong\tilde{U}_{i}$, because of the singularities of $Z\cap U_{i}$ and the quotient map $\pi_{\mathrm{an}}$, the function $\hat{\varphi}_{i}$ does not coincide with $\tilde{\varphi}_{i}$ even if $U_{i}/\!\!/G\cong\tilde{U}_{i}$ is a smooth complex manifold, see the example in \cite[Section 6]{HS20}.
This fact was overlooked in the previous version.

Since $\tilde{\varphi}_{i}$ is only a continuous function, the action of the operator $\mathrm{d}\mathrm{d}^{c}$ on $\tilde{\varphi}_{i}$ does not make sense and hence $\kappa$ can not produces a $\mathcal{C}^{\infty}$-K\"{a}hler differential 2-form.
Consequently, there is no direct definition of symplectic form on the singular symplectic quotient in the sense of Definition \ref{sym-spmd}.
However, akin to the construction of  K\"{a}hler classes in Subsection \ref{k-c-c}, we can still define the K\"{a}hler class of $\kappa$ valued in the \v{C}ech cohomology $\check{\mathrm{H}}^{2}(M_{0}, \underline{\mathbb{R}})$.
On the one hand, observe that there is a short exact sequence of sheaves on $M_{0}$:
\begin{equation}\label{s-e-c0}
\xymatrix@C=0.5cm{
  0 \ar[r] & \mathscr{PH}_{M_{0},\mathbb{R}} \ar[rr]^{} && \C^{0}_{M_{0},\mathbb{R}}
   \ar[rr]^{} && \mathscr{F}_{M_{0},\mathbb{R}} \ar[r] & 0. }
\end{equation}
From \eqref{s-e-c0}, we get a canonical morphism
\begin{equation}\label{del-c3}
\delta^{0}: \check{\mathrm{H}}^{0}(M_{0}, \mathscr{F}_{M_{0},\mathbb{R}})
\longrightarrow \check{\mathrm{H}}^{1}(M_{0}, \mathscr{PH}_{M_{0},\mathbb{R}} ).
\end{equation}
On the other hand, by the short exact sequence
\begin{equation*}
\xymatrix@C=0.5cm{
  0 \ar[r] & \underline{\mathbb{R}} \ar[rr]^{\sqrt{-1}\cdot} && \mathcal{O}_{M_{0}}
   \ar[rr]^{\mathrm{Re}\quad} && \mathscr{PH}_{M_{0},\mathbb{R}} \ar[r] & 0}
\end{equation*}
we get another canonical morphism
\begin{equation}\label{del-c4}
\delta^{1}: \check{\mathrm{H}}^{1}(M_{0}, \mathscr{PH}_{M_{0},\mathbb{R}})
\longrightarrow \check{\mathrm{H}}^{2}(M_{0}, \underline{\mathbb{R}} ).
\end{equation}
Combining \eqref{del-c3} with \eqref{del-c4} comprises to a canonical morphism
\begin{equation*}
c_{1}: \check{\mathrm{H}}^{0}(M_{0}, \mathscr{F}_{M_{0},\mathbb{R}})\longrightarrow \check{\mathrm{H}}^{2}(M_{0}, \underline{\mathbb{R}}).
\end{equation*}
We call $[\omega_{0}]:=c_{1}(\kappa)\in \check{\mathrm{H}}^{2}(M_{0}, \underline{\mathbb{R}})$ the K\"{a}hler class corresponding to $\kappa$.

The following result is a direct consequence of \cite[Corollary 4]{HS20}.
\begin{prop}\label{ii}
Let  $(M,ds^{2},G,\mu)$ be a K\"{a}hler Hamiltonian $G$-manifold and $\omega=-\mathrm{Im}\,ds^{2}$ the K\"{a}hler form.
Then there exists a unique class $[\omega_{0}]$ in the \v{C}ech cohomology $\check{\mathrm{H}}^{2}(M_{0}, \underline{\mathbb{R}})$ such that
$\pi^{\ast}([\omega_{0}])=\imath^{\ast}([\omega])$, where $[\omega]$ is considered as a cohomology class in $\check{\mathrm{H}}^{2}(M, \underline{\mathbb{R}})$ via the canonical isomorphism $\check{\mathrm{H}}^{2}(M, \underline{\mathbb{R}})\cong \mathrm{H}^{2}_{\mathrm{DR}}(M)$.
\end{prop}
\begin{proof}
Let $N=\bigcup^{l}_{\alpha=1}U_{\alpha}$ be the union of local slice models, which is an open neighborhood of $Z$.
Then the collection $\kappa=\{(U_{\alpha}, \varphi_{\alpha})\}^{l}_{\alpha=1}$ determines a class in $\check{\mathrm{H}}^{0}(N, \mathscr{K}^{1}_{N, \mathbb{R}})$; moreover, we have the K\"{a}hler class $$
[\omega_{N}]:=c_{1}(\kappa)\in\check{\mathrm{H}}^{2}(N, \underline{\mathbb{R}}).
$$
Actually, we can identify $[\omega_{N}]$ with the de Rham cohomology class represented by $\imath^{\ast}_{N}(\omega)$ via the canonical isomorphism $\check{\mathrm{H}}^{2}(N, \underline{\mathbb{R}})\cong \mathrm{H}^{2}_{\mathrm{DR}}(N)$, where $\imath_{N}$ the inclusion of $N$ into $M$.
Owing to \cite[Corollary 4]{HS20}, the K\"{a}hler metric $\kappa$ determines a unique K\"{a}hler metric with continuous local potentials $\tilde{\kappa}:=\{(\tilde{U}_{\alpha}, \tilde{\varphi}_{\alpha})\}^{l}_{\alpha=1}$ on $M_{0}$ satisfying $\imath^{\ast}(\varphi_{\alpha})=\pi^{\ast}(\tilde{\varphi}_{\alpha})$.
Let $[\omega_{0}]=c_{1}(\tilde{\kappa})\in\check{\mathrm{H}}^{2}(M_{0}, \underline{\mathbb{R}})$
be the K\"{a}hler class of $\tilde{\kappa}$.
To end the proof, it is sufficient to show
$\pi^{\ast}([\omega_{0}])=\imath^{\ast}([\omega_{N}])$.
On the one hand, without loss of generality, we can assume
$$
\tilde{\varphi}_{\alpha\beta}:=
(\tilde{\varphi}_{\beta}-\tilde{\varphi}_{\beta})|_{\tilde{U}_{\alpha\beta}}
=\tilde{f}_{\alpha\beta}+\bar{\tilde{f}}_{\alpha\beta}
$$
for some homomorphic function $\tilde{f}_{\alpha\beta}$ on $\tilde{U}_{\alpha\beta}$.
By definition, we get a \v{C}ech 1-cochain
$$
\tilde{f}=\{(\tilde{U}_{\alpha\beta}, \tilde{f}_{\alpha\beta})\}
\in\check{C}^{1}(\tilde{\mathscr{U}}, \mathcal{O}_{M_{0}}).
$$
Let $\delta$ be the \v{C}ech differential.
Then $\delta\tilde{f}$ is a \v{C}ech 2-cocycle with coefficients in $\underline{\mathbb{R}}$ which represents the K\"{a}hler class $[\omega_{0}]$.
On the other hand, let $f_{\alpha\beta}=\pi_{\mathrm{an}}^{\ast}(\tilde{f}_{\alpha\beta})$
then the family of holomorphic functions $\{f_{\alpha\beta}\}$ gives rise to a \v{C}ech 1-cochain $f\in\check{C}^{1}(\mathscr{U},\mathcal{O}_{N})$ such that $\delta f\in\check{C}^{2}(\mathscr{U},\underline{\mathbb{R}})$ is a representative of the K\"{a}hler class $[\omega_{N}]$ (cf. \cite[\S\,4.2]{Va89}).
Since $\pi^{\ast}(\tilde{f}_{\alpha\beta})=\imath^{\ast}(f_{\alpha\beta})$, we are led to the conclusion $\pi^{\ast}([\omega_{0}])=\imath^{\ast}([\omega_{N}])$ and this completes the proof.
\end{proof}

\begin{rem}
Due to the Richberg regularisation theorem on complex space (cf. \cite[Theorem 1]{Va89}),
for the collection $\tilde{\kappa}=\{(\tilde{U}_{\alpha}, \tilde{\varphi}_{\alpha})\}^{l}_{\alpha=1}$ on the singular K\"{a}hler quotient $M_{0}$, there exists a family of smooth strictly p.s.h. functions $\{\phi_{\alpha}:\tilde{U}_{\alpha}\rightarrow\mathbb{R}\}^{l}_{\alpha=1}$ such that $\phi_{\alpha}-\phi_{\beta}=\tilde{\varphi}_{\alpha}-\tilde{\varphi}_{\beta}$ on $\tilde{U}_{\alpha}\cap\tilde{U}_{\beta}\neq\emptyset$.
\end{rem}

\section{Symplectic structures on quotient spaces}\label{sym-quo}

Group actions on stratified pseudomanifolds are more complicated than actions on smooth manifolds. The reasons lie in the facts that a stratified pseudomanifold itself is singular and the group action also produces new singularities of the quotient spaces.
Assume that $X$ is a Hausdorff topological space and $G$ a topological group.
Recall that a $G$-action on $X$ is defined to be a continuous map
$\Theta:G\times X\rightarrow X$
such that
\begin{itemize}
  \item [(i)]$\Theta(g,\Theta(h,x))=\Theta(gh,x)$
             for all $g,h\in G$ and $x\in X$;
  \item [(ii)]$\Theta(e,x)=x$ for all $x\in X$,
             where $e\in G$ is the identity.
\end{itemize}
\begin{defn}\label{G-str}
Let $G$ be a compact connected Lie group.
A smooth $G$-action on a smooth stratified pseudomanifold $\mathcal{X}=(X,\mathcal{S}_{X},\mathcal{C}^{\infty}_{X})$ is defined to be a $G$-action on $X$ such that the following conditions hold:
\begin{itemize}
  \item [(i)] the map $\Theta:G\times X\rightarrow X$ is smooth in the category of  local $\mathcal{C}^{\infty}$-ringed spaces;
  \item [(ii)] the $G$-action preserves the stratification structure, i.e., every stratum of $X$ is a smooth $G$-manifold.
\end{itemize}
A smooth stratified pseudomanifold $\mathcal{X}$ together with a given smooth $G$-action is called a smooth \emph{$G$-stratified pseudomanifold}.
\end{defn}

Inspired by the definition of symplectic orbifolds, we introduce an indirect definition of a symplectic structure on the quotient of a smooth $G$-stratified pseudomanifold.
\begin{defn}\label{ind-ss}
Let $\mathcal{X}=(X,\mathcal{S}_{X},\mathcal{C}^{\infty}_{X})$ be a smooth $G$-stratified pseudomanifold.
A $G$-invariant pre-symplectic form $\omega$ on $X$ is said to define a \emph{symplectic structure} on the quotient space $X/G$, if for every stratum $S\in\mathcal{S}_{X}$ the pullback $\underline{i}^\ast\omega\in\Omega^{2}(S)$ is a pre-symplectic form on $S$ which degenerates along the $G$-orbits, where $\underline{i}=(i, i_{\sharp}):(S,\C^{\infty}_{S})\hookrightarrow (X,\C^{\infty}_{X})$ is the inclusion map.
\end{defn}

We are now in a position to prove Theorem \ref{thm1}.
\begin{proof}[Proof of Theorem \ref{thm1}]
Given a symplectic Hamiltonian $G$-manifold $(M,\omega,G,\mu)$ with dimension $2m$,
we consider $M$ as a smooth stratified pseudomanifold with the orbit type stratification.
First, we show that the level set $Z=\mu^{-1}(0)$ admits a natural structure of a presymplectic stratified pseudomanifold which induces a symplectic structure on the quotient space $M_{0}=Z/G$ in the sense of Definition \ref{ind-ss}.
Secondly, we show that the symplectic form $\omega$ determines a unique \v{C}ech cohomology class in $\check{\mathrm{H}}^{2}(M_{0}, \underline{\mathbb{R}})$ which makes $M_{0}$ into a cohomologically symplectic stratified pseudomanifold.
Finally, we show that Arms--Cushman--Gotay's Poisson bracket $\{-,-\}_{M_{0}}$ determines a natural injective morphism from the torsion-free cotangent sheaf $\Omega^{1}_{M_{0},\mathrm{tf}}$ to the tangent sheaf $\mathcal{T}_{M_{0}}$.

\textbf{Step 1.}
According to Example \ref{G-mfd}, the smooth $G$-manifold is a smooth stratified pseudomanifold with respect to the Whitney stratification $\mathcal{S}_{\mathrm{orb}}$ by orbit types:
\begin{equation*}
M=\coprod_{H<G}M_{(H)},
\end{equation*}
where $H$ ranges over all closed subgroups of $G$.
Let $M_{\mathrm{prin}}$ be the \emph{principal stratum} in $\mathcal{S}_{\mathrm{orb}}$, which is an open dense subset in $M$.
The level set $Z$ can be decomposed as
\begin{equation}\label{orbit-type-2}
Z=\coprod_{H<G}Z_{(H)},
\end{equation}
where $Z_{(H)}=Z\cap M_{(H)}$ is a $G$-invariant submanifold of $M$.
The decomposition \eqref{orbit-type-2} gives rise to a Whitney stratification $\MC_{Z}$ on $Z$; moreover, the corresponding principal stratum $Z_{\mathrm{prin}}=Z\cap M_{\mathrm{prin}}$ is an open and dense subset in $Z$.
Assume that
$$
\mathfrak{A}=
\{(U_{\lambda},\phi_{\lambda})\,|\,\lambda\in\Lambda\}
$$
is the $\C^{\infty}$-atlas of $M$.
A straightforward checking shows that
$$
\mathfrak{A}_{Z}:=
\{\phi_{\lambda}|_{Z\cap U_{\lambda}}:Z\cap U_{\lambda}\rightarrow\mathbb{R}^{2m}\,|\,\lambda\in\Lambda\}
$$
yields a $\C^{\infty}$-singular atlas of the stratified space $(Z,\mathcal{S}_{Z})$.
As a result, $(Z,\mathcal{S}_{Z},\C^{\infty}_{Z})$ becomes a smooth stratified pseudomanifold.
In particular, the inclusion mapping defines a morphism of local ringed spaces
$$
\underline{\imath}=(\imath,\imath_{\sharp}):(Z,\C^{\infty}_{Z})\longrightarrow(M,\C^{\infty}_{M}).
$$
Put $\sigma=\underline{\imath}^{\ast}\omega$.
On the one hand, it is clear that $\sigma$ is a closed $\C^{\infty}$-K\"{a}hler differential 2-form on $Z$.
On the other hand, in the setting of local $\C^{\infty}$-ringed spaces, for every stratum $Z_{(H)}$ we have the following commutative diagram of smooth maps:
\begin{equation*}
\vcenter{
\xymatrix@C=1.5cm{
  Z_{(H)} \ar[d]_{\underline{\imath}_{H}} \ar[dr]^{\underline{\jmath}_{H}}        \\
  Z \ar[r]_{\underline{\imath}\quad}  & M. }
  }
\end{equation*}
The commutativity of the diagram above means
$$
\underline{\imath}^{\ast}_{H}(\sigma)=
\underline{\imath}^{\ast}_{H}(\underline{\imath}^{\ast}\omega)=
\underline{\jmath}^{\ast}_{H}(\omega),
$$
which is a usual presymplectic form on the stratum $Z_{(H)}$.
This implies that $\sigma$ is a stratified presymplectic form on $Z$.

Consider the geometric structure of the symplecytic quotient $M_{0}=Z/G$.
Due to \cite[Theorem 2.1]{SL91}, the orbit space $M_{0}$ is a stratified space with the canonical stratification $\MC_{M_{0}}$:
\begin{equation*}
M_{0}=\coprod_{H<G}(M_{0})_{(H)},
\end{equation*}
where
$(M_{0})_{(H)}=Z_{(H)}/G$.
The orbit map $\pi:Z\rightarrow M_{0}$ is a continuous map under the quotient topology of $M_{0}$, and the principal stratum
$(M_{0})_{\mathrm{prin}}=Z_{\mathrm{prin}}/G$ is open and dense in $M_{0}$.
We now endow a smooth structure with $M_{0}$.
The structure sheaf $\C^{\infty}_{M_{0}}=(\pi_{\ast}\C^{\infty}_{Z})^{G}$ is defined by setting
\begin{equation*}
\C^{\infty}_{M_{0}}(U)=\C^{\infty}_{Z}(\pi^{-1}(U))^{G}
\end{equation*}
for any open subset $U$ in $M_{0}$.
In fact, the structure sheaf $\C^{\infty}_{M_{0}}$ comes from a canonical $\C^{\infty}$-singular atlas of $M_{0}$.
The proof is a combination of the local structure theorem of the reduced space \cite[Theorem 5.1]{SL91} and the  proper embedding theorem of orbit spaces by Schwarz \cite{Sc75}, see \cite[Section 6]{Pf01b} for details.
So that $\mathcal{M}_{0}=(M_{0},\mathcal{S}_{M_{0}},\C^{\infty}_{M_{0}})$ is a smooth stratified pseudomanifold, and the orbit map
$
\underline{\pi}=(\pi,\pi_{\sharp}):(Z,\C^{\infty}_{Z})\longrightarrow (M_{0},\C^{\infty}_{M_{0}})
$
becomes a morphism of local $\C^{\infty}$-ringed spaces.
Consequently, we obtain an \emph{inclusion-quotient diagram} in the category of smooth stratified pseudomanifolds:

\begin{equation}\label{str-in-qu}
\vcenter{
\xymatrix@=1cm{
  (Z,\mathcal{S}_{Z},\C^{\infty}_{Z}) \ar[d]_{\underline{\pi}}\, \ar[r]^{\underline{\imath}} &   (M,\mathcal{S}_{\mathrm{orb}}, \C^{\infty}_{M})     \\
  (M_{0},\mathcal{S}_{M_{0}},\C^{\infty}_{M_{0}}).}
  }
\end{equation}
Particularly, for every orbit type $(H)$, there exists a unique symplectic form $(\omega_{0})_{(H)}$ on $(M_{0})_{(H)}$ such that $\underline{\pi}^{\ast}_{H}(\omega_{0})_{(H)}
=\underline{\jmath}^{\ast}_{H}(\omega)=\underline{\imath}^{\ast}_{H}(\sigma)$,
where
$\underline{\pi}_{H}:Z_{(H)}\rightarrow(M_{0})_{(H)}$ is the quotient map (cf. \cite[Theorem 2.1]{SL91}).
It follows that the pre-symplectic form $\underline{\imath}^{\ast}_{H}(\sigma)$ degenerates along the $G$-orbits.
By definition, the $\mathcal{C}^{\infty}$-K\"{a}hler differential 2-form $\sigma$ on $Z$ defines a symplectic structure on $M_{0}$ in the sense of Definition \ref{ind-ss}.

\textbf{Step 2.}
If $0$ is a regular value of $\mu$,  then the level set $\mu^{-1}(0)$ is a smooth submanifold of $M$ and the reduced space $M_{0}$ is provided with a unique orbifold symplectic form $\omega_{0}$.
Because the de Rham cohomology of orbifold differential forms on $M_{0}$ is canonically isomorphic to the \v{C}ech cohomology $\check{\mathrm{H}}^{\ast}(M_{0}, \mathbb{R})$, the cohomology class of the orbifold symplectic form $\omega_{0}$ corresponds to a unique \v{C}ech cohomology class, denoted by $[\omega_{0}]\in\check{\mathrm{H}}^{2}(M_{0}, \underline{\mathbb{R}})$, such that $\underline{\pi}^{\ast}([\omega_{0}])=\underline{\imath}^{\ast}([\omega])$ in $\check{\mathrm{H}}^{2}(Z, \underline{\mathbb{R}})$.

Suppose $0$ is a critical value of $\mu$, we have the following indirect approach to construct the \v{C}ech cohomology class $[\omega_{0}]\in\check{\mathrm{H}}^{2}(M_{0}, \underline{\mathbb{R}})$.
On account of a result by Heinzner--Huckleberry--Loose \cite{HHL94},  the symplectic Hamiltonian $G$-manifold $(M,\omega,G,\mu)$ admits a canonical complex \lq\lq thickening"  $(X,\tau,G,\Phi)$,
called the Hamiltonian K\"{a}hlerian extension:
\begin{itemize}
  \item [(1)] $(X,\tau)$ is a K\"{a}hlerian Stein manifold on which $G$ acts by holomorphic isometries;
  \item [(2)] there exists a closed $G$-equivariant embedding $i:M\rightarrow X$ such that $i^{\ast}\tau=\omega$ and the image $i(M)$ is the fixed point set of an anti-holomorphic equivariant involution;
  \item [(3)] there is a unique moment map $\Phi:X\rightarrow\mathfrak{g}^{\ast}$ such that $\mu=i^{\ast}\Phi$.
\end{itemize}
Here the complex dimension of $X$ is equal to $\mathrm{dim}_{\mathbb{R}}\,M$. In addition, the K\"{a}hler form $\tau$ and the moment map $\Phi$ are unique up to
$G$-equivariant diffeomorphism around $M$ (cf. \cite[Theorem 5.1]{St02}).
Consider the reductions of $(M,\omega)$ and $(X,\tau)$ and then
we have the following commutative \emph{extension-inclusion-quotient} cube in the category of  local $\C^{\infty}$-ringed spaces.
\begin{equation}\label{eiq-cube}
\vcenter{
\xymatrix@=0.5cm{
  & \Phi^{-1}(0) \ar[rr]^{\underline{\jmath}} \ar'[d][dd]^{\underline{\varpi}}
      &  & X \ar[dd]^{}       \\
  \mu^{-1}(0) \ar[ur]^{\underline{i}}\ar[rr]^{\quad\quad\quad\underline{\imath}}\ar[dd]_{\underline{\pi}}
      &  & M \ar[ur]^{\underline{i}}\ar[dd]^{} \\
  & X_{0} \ar'[r][rr]^{}
      &  & X/G                \\
  M_{0} \ar[rr]^{}\ar[ur]^{\underline{k}}
      &  & M/G \ar[ur]        }
      }
\end{equation}
Let $[\tau]$ and $[\omega]$  be the \v{C}ech cohomology classes in $\check{\mathrm{H}}^{2}(X, \underline{\mathbb{R}})$ and $\check{\mathrm{H}}^{2}(M, \underline{\mathbb{R}})$ which correspond to the de Rham cohomology classes represented by $\tau$ and $\omega$ respectively.
Since $i^{\ast}\tau=\omega$, we get $i^{\ast}([\tau])=[\omega]$ in $\check{\mathrm{H}}^{2}(M, \underline{\mathbb{R}})$.
It follows from Proposition \ref{ii} that there exists a unique \v{C}ech cohomology class $[\tau_{0}]\in\check{\mathrm{H}}^{2}(X_{0}, \underline{\mathbb{R}})$ such that
$\varpi^{\ast}([\tau_{0}])=\jmath^{\ast}([\tau])$ in $\check{\mathrm{H}}^{2}(\Phi^{-1}(0), \underline{\mathbb{R}})$.
Set $[\omega_{0}]=k^{\ast}([\tau_{0}])$, which is a \v{C}ech cohomology class in  $\check{\mathrm{H}}^{2}(M_{0}, \underline{\mathbb{R}})$.
The commutative diagram \eqref{eiq-cube} induces a commutative diagram of \v{C}ech cohomology groups:
\begin{equation}\label{ccech}
\vcenter{
\xymatrix{
  \check{\mathrm{H}}^{2}(X_{0}, \underline{\mathbb{R}}) \ar[d]_{k^{\ast}} \ar[r]^{\varpi^{\ast}} & \check{\mathrm{H}}^{2}(\Phi^{-1}(0), \underline{\mathbb{R}}) \ar[d]_{i^{\ast}}  & \ar[l]_{\quad\quad\jmath^{\ast}} \check{\mathrm{H}}^{2}(X, \underline{\mathbb{R}}) \ar[d]^{i^{\ast}} \\
  \check{\mathrm{H}}^{2}(M_{0}, \underline{\mathbb{R}}) \ar[r]^{\pi^{\ast}} & \check{\mathrm{H}}^{2}(\mu^{-1}(0), \underline{\mathbb{R}})  & \ar[l]_{\quad\quad\imath^{\ast}}\check{\mathrm{H}}^{2}(M, \underline{\mathbb{R}})   }
  }
\end{equation}
From the commutativity of  \eqref{ccech}, we get
 $\pi^{\ast}([\omega_{0}])=\imath^{\ast}([\omega])$ in $\check{\mathrm{H}}^{2}(\mu^{-1}(0), \underline{\mathbb{R}})$.
It is worth noting that $0$ may be a critical value of $\Phi$ even if it is a regular value of $\mu$, and therefore the level set $\Phi^{-1}(0)$ may be a singular space.
However, when $0$ is a regular value of $\mu$ the \v{C}ech cohomology class determined by the orbifold symplectic form on $M_{0}$ coincides with the one constructed indirectly using the Hamiltonian K\"{a}hlerian extension.

Let $2n=\mathrm{dim}\,M_{0}$.
It remains to show that the \v{C}ech cohomology class $[\omega^{n}_{0}]\in\check{\mathrm{H}}^{2n}(M_{0}, \underline{\mathbb{R}})$ is nonzero.
In \cite{Sj05} Sjamaar introduced a de Rham model\footnote{
A differential form on $M_{0}$ in the sense of Sjamaar is defined to be a usual differential form $\alpha$ on the top stratum $(M_{0})_{\mathrm{prin}}$ such that $\pi^{\ast}_{\mathrm{prin}}\alpha=\imath^{\ast}_{\mathrm{prin}}\tilde{\alpha}$ for some usual differential form $\tilde{\alpha}$ on $M$, where $\pi_{\mathrm{prin}}:Z_{\mathrm{prin}}\rightarrow(M_{0})_{\mathrm{prin}}$ is the quotient map and $\imath_{\mathrm{prin}}:Z_{\mathrm{prin}}\rightarrow M$ is the inclusion map.}
for $M_{0}$ whose cohomology ring is isomorphic to the \v{C}ech cohomology ring $\check{\mathrm{H}}^{\ast}(M_{0},\underline{\mathbb{R}})$.
Owing to \cite[Theorem 5.5]{Sj05}, the \v{C}ech cohomology class $[\omega_{0}]\in\check{\mathrm{H}}^{2}(M_{0},\underline{\mathbb{R}})$ is identical to the class represented by the symplectic form $\omega_{\mathrm{prin}}$ on $M_{0}$ introduced in \cite[Section 3]{Sj05}.
As a direct consequence of \cite[Corollary 7.6.]{Sj05}, we are led to the conclusion that $[\omega_{0}]$ yields a cohomologically symplectic structure on $M_{0}$.

\textbf{Step 3.}
Recall the construction of universal reductions by Arms--Cushman--Gotay \cite[Section 2]{ACG89}.
Set $\Pi:M\twoheadrightarrow M/G$ the quotient map.
Observe that the symplectic form $\omega$ determines a regular Poisson bracket, denoted by $\{-,-\}_{\omega}$, on $\C^{\infty}(M)$.
This enables us to define a Poisson bracket $\{-,-\}_{M/G}$ on $\C^{\infty}(M/G)=\C^{\infty}(M)^{G}$ by setting
$$
\{h_{1},h_{2}\}_{M/G}(\Pi(x))=\{\Pi^{\ast}h_{1},\Pi^{\ast}h_{2}\}_{\omega}(x)
$$
for any $x$ in $M$.
Let $\mathcal{I}^{G}_{Z}$ be the ideal of $G$-invariant smooth functions vanishing on $Z$.
Then we have
$$
\C^{\infty}(M_{0})=\C^{\infty}(Z)^{G}\cong\C^{\infty}(M)^{G}/\mathcal{I}^{G}_{Z}.
$$
For each pair of smooth functions $f_{1},f_{2}\in\C^{\infty}(M_{0})$, there exist smooth extensions, denoted by $\tilde{f}_{1}$ and $\tilde{f}_{2}$, to the orbit space $M/G$.
Define the bracket
$$
\{f_{1},f_{2}\}_{M_{0}}:=\{\tilde{f}_{1},\tilde{f}_{2}\}_{M/G}\big|_{M_{0}}.
$$
Owing to \cite[Theorem 1]{ACG89}, we know that $\{-,-\}_{M_{0}}$ is a Poisson bracket.
In particular, the Poisson algebra $(\C^{\infty}(M_{0}),\{-,-\}_{M_{0}})$ is \emph{non-degenerate}, i.e., the Casimir elements in $\C^{\infty}(M_{0})$ are only locally constant functions, see \cite[Theorem 2]{ACG89} or \cite[Corollary 5.11]{SL91}.
Moreover, the bracket $\{-,-\}_{M_{0}}$ is compatible with the stratum-wise symplectic forms on $M_{0}$ (cf. \cite[Proposition 3.1]{SL91}).

We will prove the assertion (iii).
From now on, we assume that $U$ is an arbitrary open subset in $M_{0}$ and let $U_{\mathrm{prin}}=U\cap(M_{0})_{\mathrm{prin}}$ be the principle stratum of $U$ under the induced stratification structure from $M_{0}$.
We can find a $G$-invariant open neighborhood $\mathcal{W}$ of $\pi^{-1}(U)$ in $M$ such that $\pi^{-1}(U)=Z\cap\mathcal{W}$.
Set $\mathcal{W}_{\mathrm{prin}}=\mathcal{W}\cap M_{\mathrm{prin}}$, which is a $G$-invariant open subset in $M$ having the \lq\lq smallest" orbit type.
Let $i:\mathcal{W}_{\mathrm{prin}}\hookrightarrow M$ be the inclusion map.
Then we get a Hamiltonian $G$-open subset
$(\mathcal{W}_{\mathrm{prin}},\omega_{\mathrm{prin}},G,\mu_{\mathrm{prin}})$,
where $\omega_{\mathrm{prin}}=i^{\ast}\omega$ and $\mu_{\mathrm{prin}}$ represents the restriction of $\mu$ to $\mathcal{W}_{\mathrm{prin}}$.
On account of \cite[Theorem 2.1]{SL91}, the level set $\mu^{-1}_{\mathrm{prin}}(0)=Z\cap\mathcal{W}_{\mathrm{prin}}$ is a smooth manifold, and we have $\mu_{\mathrm{prin}}^{-1}(0)/G=U_{\mathrm{prin}}$.
In particular, the Marsden--Weinstein reduction is well-defined.
Denote by $\imath_{\mathrm{prin}}:\mu_{\mathrm{prin}}^{-1}(0)\hookrightarrow\mathcal{W}_{\mathrm{prin}}$ the inclusion map and $\pi_{\mathrm{prin}}:\mu_{\mathrm{prin}}^{-1}(0)\twoheadrightarrow U_{\mathrm{prin}}$ the orbit map.
Then there exists a natural symplectic form $(\omega_{0})_{\mathrm{prin}}$ on $U_{\mathrm{prin}}$ satisfying the condition
$$
\imath^{\ast}_{\mathrm{prin}}\omega_{\mathrm{prin}}
=\pi^{\ast}_{\mathrm{prin}}\bigl((\omega_{0})_{\mathrm{prin}}\bigr).
$$
Observe that $(U_{\mathrm{prin}},(\omega_{0})_{\mathrm{prin}})$ is a symplectic manifold.
The induced morphism of $\C^{\infty}(U_{\mathrm{prin}})$-modules
\begin{equation}\label{u-omg-prin}
  (\omega_{0})_{\mathrm{prin}}^{\flat}:\mathcal{T}_{M_{0}}(U_{\mathrm{prin}})
  \longrightarrow \Omega^{1}_{M_{0}}(U_{\mathrm{prin}})
\end{equation}
is an isomorphism.

Consider the universal reduction for $(\mathcal{W},\omega|_{\mathcal{W}},G,\mu|_{\mathcal{W}})$, and then we obtain a natural Poisson bracket $\{-,-\}_{U}$ on $\C^{\infty}_{M_{0}}(U)$.
Define the map
\begin{eqnarray*}
  \chi(U):\C^{\infty}_{M_{0}}(U) &\longrightarrow& \mathcal{T}_{M_{0}}(U) \\
  f&\longmapsto& X_{f}:=\{-,f\}_{U},
\end{eqnarray*}
which defines a sheaf morphism $\chi:\C^{\infty}_{M_{0}}\rightarrow\mathcal{T}_{M_{0}}$.
Since our smooth stratified pseudomanifolds are subcaresian spaces, it follows from \cite[Theorem 3.3]{CS20} that the derivation $X_{f}$ is a $\C^{\infty}$-derivation.
We claim that $\chi(U)$ is a $\C^{\infty}$-derivation of the $\C^{\infty}$-ring $\C^{\infty}_{M_{0}}(U)$ into the  $\C^{\infty}_{M_{0}}(U)$-module $\mathcal{T}_{M_{0}}(U)$ (see Definition \ref{c-inf-der}).
Let $h:\mathbb{R}^{n}\rightarrow\mathbb{R}$ be an arbitrary smooth function.
For any $(f_{1},\cdots,f_{n})\in(\C^{\infty}_{M_{0}}(U))^{n}$, define the function
$$(\Phi_{h}(f_{1},\cdots,f_{n}))(x)=h(f_{1}(x),\cdots,f_{n}(x))$$
where $x\in U$.
By definition, $\Phi_{h}(f_{1},\cdots,f_{n})$ is a smooth function on $U$.
For an arbitrary smooth function $g\in\C^{\infty}_{M_{0}}(U)$, since $X_{g}$ is a $\C^{\infty}$-derivation we have
\begin{eqnarray*}
  X_{\Phi_{h}(f_{1},\cdots,f_{n})}(g)
   &=& \{g,\Phi_{h}(f_{1},\cdots,f_{n})\}_{U} \\
   &=& -\{\Phi_{h}(f_{1},\cdots,f_{n}),g\}_{U} \\
   &=& -X_{g}(\Phi_{h}(f_{1},\cdots,f_{n})) \\
   &=& -\sum^{n}_{i=1}\Phi_{\frac{\partial h}{\partial x_{i}}}(f_{1},\cdots,f_{n}))\cdot X_{g}(f_{i}) \\
   &=& \sum^{n}_{i=1}\Phi_{\frac{\partial h}{\partial x_{i}}}(f_{1},\cdots,f_{n})\cdot \{g,f_{i}\}_{U} \\
   &=&  \sum^{n}_{i=1}\Phi_{\frac{\partial h}{\partial x_{i}}}(f_{1},\cdots,f_{n})\cdot X_{f_{i}}(g).
\end{eqnarray*}
This implies
$X_{\Phi_{h}(f_{1},\cdots,f_{n})}=\sum^{n}_{i=1}\Phi_{\frac{\partial h}{\partial x_{i}}}(f_{1},\cdots,f_{n})\cdot X_{f_{i}}$
and therefore the claim holds.
On account of the universal property of the cotangent module, there exists a unique $\C^{\infty}_{M_{0}}(U)$-linear map
$$
\chi^{\sharp}(U):\Omega^{1}_{M_{0}}(U)\longrightarrow\mathcal{T}_{M_{0}}(U)
$$
such that the following diagram is commutative:
\begin{equation*}
\xymatrix@=1cm{
  \C^{\infty}_{M_{0}}(U) \ar[d]_{\mathrm{d}} \ar[r]^{\chi(U)} &   \mathcal{T}_{M_{0}}(U)    \\
  \Omega^{1}_{M_{0}}(U) \ar[ur]_{\chi^{\sharp}(U)}.  }
\end{equation*}
So we get a sheaf morphism $\chi^{\sharp}:\Omega^{1}_{M_{0}}\rightarrow\mathcal{T}_{M_{0}}$ together with the commutative diagram
\begin{equation*}
\xymatrix@=1cm{
  \C^{\infty}_{M_{0}} \ar[d]_{\mathrm{d}} \ar[r]^{\chi} &   \mathcal{T}_{M_{0}}    \\
  \Omega^{1}_{M_{0}} \ar[ur]_{\chi^{\sharp}}.}
\end{equation*}
We will show that the kernel of $\chi^{\sharp}(U)$ is equal to the set of torsion elements.
As the Lie group $G$ is compact and $\mu_{\mathrm{prin}}^{-1}(0)$ is closed in $\mathcal{W}_{\mathrm{prin}}$,  by \cite[Theorem 3]{ACG89}, the Marsden--Weinstein reduction of $(\mathcal{W}_{\mathrm{prin}},\omega_{\mathrm{prin}},G,\mu_{\mathrm{prin}})$ coincides with its universal reduction.
It follows that the Poisson bracket $\{-,-\}_{U_{\mathrm{prin}}}$ on $\C^{\infty}_{M_{0}}(U_{\mathrm{prin}})$ is actually determined by the symplectic form $(\omega_{0})_{\mathrm{prin}}$.
As a result, the morphism
$$
\chi^{\sharp}(U_{\mathrm{prin}}):\Omega^{1}_{M_{0}}(U_{\mathrm{prin}})
\longrightarrow\mathcal{T}_{M_{0}}(U_{\mathrm{prin}})
$$
is the inverse of \eqref{u-omg-prin}.
Consider the commutative square
\begin{equation}\label{pi-prin}
\vcenter{
\xymatrix@C=1.5cm{
  \Omega^{1}_{M_{0}}(U) \ar[d]_{\rho^{1}_{\mathrm{prin}}}
  \ar[r]^{\chi^{\sharp}(U)} & \mathcal{T}_{M_{0}}(U) \ar[d]^{\rho_{\mathrm{prin}}} \\
   \Omega^{1}_{M_{0}}(U_{\mathrm{prin}})\ar[r]^{\chi^{\sharp}(U_{\mathrm{prin}})}_{\simeq} & \mathcal{T}_{M_{0}}(U_{\mathrm{prin}}).}
   }
\end{equation}
Note that $\rho_{\mathrm{prin}}$ in \eqref{pi-prin} is injective.
The commutativity of  \eqref{pi-prin} implies
\begin{equation*}\label{ker-pi}
\mathrm{Ker}\,(\chi^{\sharp}(U))=
\mathrm{Ker}\,(\rho^{1}_{\mathrm{prin}})=
\Omega^{1}_{M_{0},\mathrm{t}}(U).
\end{equation*}
Recall that
$$\Omega^{1}_{M_{0},\mathrm{tf}}(U)=\Omega^{1}_{M_{0}}(U)/\Omega^{1}_{M_{0},\mathrm{t}}(U).$$
So that $\chi^{\sharp}(U)$ determines an injective morphism, denoted by $\chi^{\ddag}(U)$, from $\Omega^{1}_{M_{0},\mathrm{tf}}(U)$ to $\mathcal{T}_{M_{0}}(U)$.
This implies an injective sheaf morphism
\begin{equation*}\label{inj-tf}
\chi^{\ddag}:\Omega^{1}_{M_{0},\mathrm{tf}}\longrightarrow\mathcal{T}_{M_{0}}
\end{equation*}
and the proof is complete.
\end{proof}

\begin{rem}
From Example \ref{G-mfd} and the inclusion-quotient diagram \eqref{str-in-qu}, we get a commutative square
\begin{equation*}
\vcenter{
\xymatrix@C=1.3cm{
  (Z,\mathcal{S}_{Z},\C^{\infty}_{Z}) \,\ar[d]_{\underline{\pi}} \ar[r]^{\underline{\imath}} & (M,\mathcal{S}_{\mathrm{orb}},\C^{\infty}_{M}) \ar[d]^{\underline{\Pi}} \\
  (M_{0},\mathcal{S}_{M_{0}},\C^{\infty}_{M_{0}})\, \ar[r]^{\underline{\jmath}} & (M/G,\mathcal{S}_{M/G}, \C^{\infty}_{M/G}).}
  }
\end{equation*}
Although, the symplectic form $\omega$ determines, and
conversely is determined by, the Poisson bracket $\{-,-\}_{\omega}$ on $M$, they have subtle different properties respectively.
The symplectic form $\omega$ can not descends to a $\C^{\infty}$-K\"{a}hler differential 2-form on $M/G$.
However, the pullback of $\omega$ to $Z$ gives rise to a presymplectic structure which produces
an indirect symplectic structure on $M_{0}$ in the sense of Definition \ref{ind-ss}.
Differing from $\omega$, the bracket $\{-,-\}_{\omega}$ induces a Poisson structure on the orbit space $M/G$ and then restricts to a non-degenerate Poisson bracket on $\mathcal{C}^{\infty}(M_{0})$.
\end{rem}

\appendix
\section{Review of $\mathcal{C}^{\infty}$-ringed spaces}\label{alg-pre}

The purpose of this appendix is to present a brief review of basic notions in the theory of $\mathcal{C}^{\infty}$-ringed spaces, and we will use \cite[Chapters 2--5]{Jo19} as our main reference.

Let $\mathfrak{C}$ be a set.
For any $n\geq0$ define
$\mathfrak{C}^{n}=\underbrace{\mathfrak{C}\times\cdots\times\mathfrak{C}}_{n-\mathrm{copies}}$
and $\mathfrak{C}^{0}=\{\emptyset\}$ by convention.

\begin{defn}
A $\mathcal{C}^{\infty}$-ring is defined to be a set $\mathfrak{C}$ together with $n$-fold operations
$$
\Phi_{f}:\mathfrak{C}^{n}\longrightarrow\mathfrak{C}
$$
for all $n\geq0$ and smooth function $f\in\mathcal{C}^{\infty}(\mathbb{R}^{n})$ such that
\begin{itemize}
  \item [(i)] For any smooth maps $F=(f_{1},\cdots,f_{m}):\mathbb{R}^{n}\rightarrow\mathbb{R}^{m}$ and $g:\mathbb{R}^{m}\rightarrow\mathbb{R}$, we have
      $$
      \Phi_{h}(c_{1},\cdots,c_{n})=
      \Phi_{g}(\Phi_{f_{1}}(c_{1},\cdots,c_{n}),\cdots,\Phi_{f_{m}}(c_{1},\cdots,c_{n}))
      $$
      where $h=g\circ F$ and $(c_{1},\cdots,c_{n})\in\mathfrak{C}^{n}$.
  \item [(ii)]For any $1\leq j\leq n$, let $\pi_{j}:\mathbb{R}^{n}\rightarrow\mathbb{R}$ be the projection on the $j$-th component, then we have
      $$\Phi_{\pi_{j}}(c_{1},\cdots,c_{n})=c_{j}.$$
\end{itemize}
\end{defn}

By a \emph{morphism} between two $\mathcal{C}^{\infty}$-rings $\bigl(\mathfrak{C},(\Phi_{f})_{f\in\mathcal{C}^{\infty}(\mathbb{R}^{n})}\bigr)$ and
$\bigl(\mathfrak{D},(\Psi_{f})_{f\in\mathcal{C}^{\infty}(\mathbb{R}^{n})}\bigr)$, we mean a map
$\phi:\mathfrak{C}\rightarrow\mathfrak{D}$ with
$$
\Psi_{f}(\phi(c_{1}),\cdots,\phi(c_{n}))=\phi\circ\Phi_{f}(\phi(c_{1}),\cdots,\phi(c_{n})).
$$
It is worthwhile to point out that every $\mathcal{C}^{\infty}$-ring $\mathfrak{C}$ has an underlying commutative $\mathbb{R}$-algebra:
\begin{itemize}
  \item The \emph{addition} is given by $c_{1}+c_{2}=\Phi_{f}(c_{1},c_{2})$, where  $f:\mathbb{R}^{2}\ni(x,y)\mapsto x+y\in\mathbb{R}$.
  \item The \emph{multiplication} is given by $c_{1}\cdot c_{2}=\Phi_{g}(c_{1},c_{2})$, where
      $g:\mathbb{R}^{2}\ni(x,y)\mapsto xy\in\mathbb{R}$.
  \item For any $\lambda\in\mathbb{R}$, the \emph{scalar multiplication} is given by $\lambda c=\Phi_{\bar{\lambda}}(c)$, where  $\bar{\lambda}:\mathbb{R}\ni x\mapsto \lambda x\in\mathbb{R}$.
  \item The elements $0$ and $1$ in $\mathfrak{C}$ are given by $0=\Phi_{\bar{0}}(\emptyset)$ and $1=\Phi_{\bar{1}}(\emptyset)$ respectively, where  $\bar{0}:\emptyset\mapsto 0$ and $\bar{1}:\emptyset\mapsto 1$.
\end{itemize}
A $\mathcal{C}^{\infty}$-ring $\mathfrak{C}$ is \emph{local} if the underlying commutative $\mathbb{R}$-algebra is local.
An \emph{ideal} in a $\mathcal{C}^{\infty}$-ring $\mathfrak{C}$ is defined to be an ideal $\mathcal{I}$ in the underlying commutative $\mathbb{R}$-algebra; in particular, the quotient $\mathfrak{C}/\mathcal{I}$ admits a natural structure of $\mathcal{C}^{\infty}$-ring (cf. \cite[Definition 2.7]{Jo19}).

\begin{defn}\label{f-g-c-r}
A $\mathcal{C}^{\infty}$-ring $\mathfrak{C}$ is \emph{finitely generated} if for any $c\in\mathfrak{C}$ there exists a smooth function $f\in\mathcal{C}^{\infty}(\mathbb{R}^{n})$ such that $c=\Phi_{f}(c_{1},\cdots,c_{n})$ for some $(c_{1},\cdots,c_{n})\in\mathfrak{C}^{n}$.
\end{defn}

The ring of smooth functions $\mathcal{C}^{\infty}(\mathbb{R}^{n})$ is the free $\mathcal{C}^{\infty}$-ring with $n$ generators, and a $\mathcal{C}^{\infty}$-ring $\mathfrak{C}$ is finitely generated if and only if $\mathfrak{C}$ is isomorphic to the $\mathcal{C}^{\infty}$-ring $\mathcal{C}^{\infty}(\mathbb{R})/\mathcal{I}$.
A striking difference between the conventional algebraic geometry and the $\mathcal{C}^{\infty}$-algebraic geometry lies in the fact that $\mathcal{C}^{\infty}(\mathbb{R}^{n})$ is not a noetherian ring.

\begin{defn}
Let $\mathfrak{C}$ be a $\mathcal{C}^{\infty}$-ring.
We call $M$ a \emph{$\mathfrak{C}$-module}, if it is a module over the underlying commutative $\mathbb{R}$-algebra of $\mathfrak{C}$.
\end{defn}
We say that a $\mathfrak{C}$-module $M$ is \emph{finitely generated} if there exists an exact sequence of $\mathfrak{C}$-modules
$$
\xymatrix@C=0.5cm{
    \mathfrak{C}\otimes\mathbb{R}^{n} \ar[r]^{} & M \ar[r] & 0. }
$$
\begin{defn}\label{c-inf-der}
Let $\mathfrak{C}$ be a $\mathcal{C}^{\infty}$-ring, $M$ a $\mathfrak{C}$-module.
A \emph{$\mathcal{C}^{\infty}$-derivation} of $\mathfrak{C}$ into $M$ is an $\mathbb{R}$-linear mapping $\mathrm{d}:\mathfrak{C}\rightarrow M$ satisfying
$$
\mathrm{d}\Phi_{f}(c_{1},\cdots,c_{n})=
\sum^{n}_{i=1}\Phi_{\frac{\partial f}{\partial x_{i}}}(c_{1},\cdots,c_{n})\cdot \mathrm{d}c_{i},
$$
for any smooth function $f\in\mathcal{C}^{\infty}(\mathbb{R}^{n})$ and $(c_{1},\cdots,c_{n})\in\mathfrak{C}^{n}$.
\end{defn}

We denote by  $\mathrm{Der}(\mathfrak{C},M)$
the space of $\mathcal{C}^{\infty}$-derivations on $\mathfrak{C}$ with values in $M$.
\begin{defn}\label{kah-diff}
Given a $\mathcal{C}^{\infty}$-ring $\mathfrak{C}$, let $\Omega_{\mathfrak{C}}$ be the quotient of the free $\mathfrak{C}$-module over the symbols $\mathrm{d}c$ for $c\in\mathfrak{C}$ by the $\mathfrak{C}$-submodule generated by all expressions of the forms
$$
\mathrm{d}\Phi_{f}(c_{1},\cdots,c_{n})-
\sum^{n}_{i=1}\Phi_{\frac{\partial f}{\partial x_{i}}}(c_{1},\cdots,c_{n})\cdot \mathrm{d}c_{i},
$$
for all $f\in\mathcal{C}^{\infty}(\mathbb{R}^{n})$ and $(c_{1},\cdots,c_{n})\in\mathfrak{C}^{n}$.
The pair $(\Omega_{\mathfrak{C}}, \mathrm{d})$ is called the \emph{cotangent module} of $\mathfrak{C}$.
\end{defn}
By definition, the cotangent module $(\Omega_{\mathfrak{C}}, \mathrm{d})$ satisfies the following universal property:
for any $\mathfrak{C}$-module $M$ and any $\mathcal{C}^{\infty}$-derivation $\delta\in\mathrm{Der}(\mathfrak{C},M)$
there exists a \emph{unique} morphism of $\mathfrak{C}$-modules
$\iota(\delta):\Omega_{\mathfrak{C}}\rightarrow M$ such that $\delta=\iota(\delta)\circ \mathrm{d}$, i.e., the diagram
\begin{equation*}
\xymatrix@=1.3cm{
  \mathfrak{C} \ar[d]_{\mathrm{d}} \ar[r]^{\delta} &   M     \\
  \Omega_{\mathfrak{C}} \ar[ur]_{\iota(\delta)}                     }
\end{equation*}
is commutative.
Moreover, there exists a canonical isomorphism
$$
\mathrm{Der}(\mathfrak{C},M)\cong
\Hom_{\mathfrak{C}}(\Omega_{\mathfrak{C}},M).
$$

Recall the definition of $\mathcal{C}^{\infty}$-ringed space.
\begin{defn}({\cite[Definition 4.8]{Jo19}})\label{inf-r-sps}
A \emph{$\mathcal{C}^{\infty}$-ringed space} is a topological space $X$ together with a sheaf $\mathcal{O}_{X}$ of $\mathcal{C}^{\infty}$-rings on $X$.
We call $(X,\mathcal{O}_{X})$ a \emph{local $\mathcal{C}^{\infty}$-ringed space},
if the stalk of the structure sheaf $\mathcal{O}_{X,x}$ is a local $\mathcal{C}^{\infty}$-ring for any $x\in X$.
\end{defn}

A special example of local $\mathcal{C}^{\infty}$-ringed space is the smooth subcartesian spaces, see Appendix \ref{smo-sub} below.
A \emph{morphism} of $\mathcal{C}^{\infty}$-ringed spaces $(X,\mathcal{O}_{X})$ and $(Y,\mathcal{O}_{Y})$ is a pair $\underline{f}=(f,f_{\sharp})$ consists of a continuous map of topological spaces $f:X\rightarrow Y$ together with a morphism of sheaves of $\mathcal{C}^{\infty}$-rings
$f_{\sharp}:\mathcal{O}_{Y}\rightarrow f_{\ast}\mathcal{O}_{X}$.
All $\mathcal{C}^{\infty}$-ringed spaces form a category which contains the category of smooth manifolds as a full subcategory.


For a smooth manifold,
the tangent bundle and differential forms play significant roles in the geometric study.
All these notions have their translations in the theory of $\mathcal{C}^{\infty}$-ringed spaces;
of course things are not well behaved for the appearance of possible singularities.
On the one hand, we use the language of sheaf to achieve such translations rather than vector bundles.
On the other hand, taking a dual point of view, we begin with the cotangent sheaf rather than the tangent sheaf.

Suppose $(X,\mathcal{O}_{X})$ is a $\mathcal{C}^{\infty}$-ringed space.
Let $U\subset X$ be an arbitrary open subset.
Then we have a $\mathcal{C}^{\infty}$-ring $\mathcal{O}_{X}(U)$ together with the cotangent module $(\Omega_{\mathcal{O}_{X}(U)}, \mathrm{d})$ in the sense of Definition \ref{kah-diff}.
So we can define a presheaf $\Omega^{1}_{X,\mathrm{pre}}$ of $\mathcal{O}_{X}$-modules by setting
$$
U\longmapsto\Omega_{\mathcal{O}_{X}(U)}
$$
for each open subset $U\subset X$.
\begin{defn}({\cite[Defition 5.29]{Jo19}})
The \emph{cotangent sheaf} $\Omega^{1}_{X}$ of a $\mathcal{C}^{\infty}$-ringed space $(X,\mathcal{O}_{X})$ is defined to be the sheafification of $\Omega^{1}_{X,\mathrm{pre}}$.
\end{defn}
Having shown the cotangent sheaf, we define the \emph{tangent sheaf} of $(X,\mathcal{O}_{X})$ to be the dual of the cotangent sheaf:
$$
\mathcal{T}_{X}=
    \Hom_{\mathcal{O}_{X}}(\Omega^{1}_{X},\mathcal{O}_{X})
$$
which is also a sheaf of $\mathcal{O}_{X}$-modules.

Let $(X,\mathcal{O}_{X})$ be a local $\mathcal{C}^{\infty}$-ringed space.
For each integer $p\geq2$ we can define a presheaf $\Omega^{p}_{X,\mathrm{pre}}$ on $X$ by assigning an $\mathcal{O}_{X}(V)$-module
$$
\Omega^{p}_{X,\mathrm{pre}}(V):=
\bigwedge^{p}\bigl(\Omega^{1}_{X}(V)\bigr).
$$
to an arbitrary open subset $V\subset X$.
The sheafification of the presheaf $\Omega^{p}_{X,\mathrm{pre}}$, denoted by $\Omega^{p}_{X}$, is called the sheaf  of $\mathcal{C}^{\infty}$-K\"{a}hler differential $p$-forms.
On account of \cite[\S\,8]{Ler22},
the $\mathcal{C}^{\infty}$-derivation $\mathrm{d}:\mathcal{O}_{X}(U)\rightarrow\Omega_{\mathcal{O}_{X}(U)}$ extends to a sheaf morphism
$\mathrm{d}:\Omega^{\ast}_{X}\rightarrow\Omega^{\ast+1}_{X}$ satisfying $\mathrm{d}\circ\mathrm{d}=0$, and this yields a sheaf complex
$(\Omega^{\bullet}_{X},\mathrm{d})$ called the \emph{$\mathcal{C}^{\infty}$-algebraic de Rham complex} of $(X,\mathcal{O}_{X})$:
\begin{equation}\label{de-com-X}
\xymatrix{
  0 \ar[r] & \mathcal{O}_{X} \ar[r]^{\mathrm{d}} & \Omega^{1}_{X} \ar[r]^{\mathrm{d}} & \Omega^{2}_{X} \ar[r]^{\mathrm{d}} & \cdots \ar[r]^{\mathrm{d}} & \Omega^{k}_{X} \ar[r] & \cdots. }
\end{equation}
The \emph{$\mathcal{C}^{\infty}$-algebraic de Rham cohomology} of $(X,\mathcal{O}_{X})$, denoted by $\mathrm{H}^{\ast}_{\mathrm{DR}}(X)$, is defined to be the hypercohomology of the sheaf complex \eqref{de-com-X}, i.e.,
$$
\mathrm{H}^{\ast}_{\DR}(X)=
\mathbb{H}^{\ast}(X,\Omega^{\bullet}_{X}).
$$
\begin{rem}
If $X$ is a usual smooth manifold and $\mathcal{O}_{X}$ equals the sheaf of smooth functions $\mathcal{C}^{\infty}_{X}$, then the cotangent sheaf $\Omega^{1}_{X}$ coincides with the sheaf of smooth sections of the cotangent bundle $T^{\ast}X$ (cf. \cite[Example 5.4]{Jo19}), and hence the tangent sheaf $\mathcal{T}_{X}$ is just the sheaf of smooth sections of the tangent bundle $TX$.
Particularly, the $\mathcal{C}^{\infty}$-algebraic de Rham cohomology of $(X,\mathcal{O}_{X})$ is equal to its usual de Rham cohomology (cf. \cite[Lemma 8.6]{Ler22}).
\end{rem}

\section{Smooth subcartesian spaces}\label{smo-sub}
In this appendix we review the definition of smooth subcartesian structure and the notion of vector pseudobundle.
Here we follow the terminologies in \cite{Ma75}.

\begin{defn}\label{ssp}
Let $S$ be a Hausdorff space.
A \emph{$\C^{\infty}$-atlas} on $S$ is a set of local homeomorphisms into the Euclidean spaces
$$
\mathfrak{A}=\{\phi_{\lambda}:U_{\lambda}\rightarrow\mathbb{R}^{n_{\lambda}}\,|\,\lambda\in\Lambda\}
$$
satisfying:
\begin{itemize}
  \item [(A1)] $\{U_{\lambda}\,|\,\lambda\in\Lambda\}$ forms an open covering of $S$;
  \item [(A2)] (Compatible condition) For any $\lambda,\lambda^{\prime}\in\Lambda$ and any $x\in U_{\lambda}\cap U_{\lambda^{\prime}}$, there exist $\C^{\infty}$-mappings $s$ extending $\phi_{\lambda^{\prime}}\circ\phi_{\lambda}^{-1}$ in an open neighborhood of $\phi_{\lambda}(x)$ in $\mathbb{R}^{\lambda}$ and $t$ extending $\phi_{\lambda}\circ\phi_{\lambda^{\prime}}^{-1}$ in an open neighborhood of $\phi_{\lambda^{\prime}}(x)$ in $\mathbb{R}^{\lambda^{\prime}}$.
\end{itemize}
A Hausdorff topological space with a $\C^{\infty}$-atlas is called a \emph{smooth subcartesian space}.
\end{defn}

Having introduced the smooth structure on $S$, one can define the smooth functions on $S$ and the smooth maps between two smooth subcartesian spaces in a natural way (cf. \cite[Section 1]{Ma75}).
In particular, all smooth subcartesian spaces form a category which contains the category of smooth manifolds as a full subcategory.

Let $(B,\mathfrak{A}_{B})$ be a smooth subcartesian spaces.
A \emph{$\C^{\infty}$-family} of $\mathbb{R}$-vector spaces is a smooth subcartesian spaces
$(E,\mathfrak{A}_{E})$
together with a smooth surjective map
$\pi:E\rightarrow B$
such that for every point $b$ in $B$ the fiber $\pi^{-1}(b)$ is a $\mathbb{R}$-vector space and the vector operators
$$
+:E\times_{B}E
\longrightarrow E\,\,\,\,
\mathrm{and}\,\,\,\,
\cdot:\mathbb{R}\times E\longrightarrow E
$$
are smooth maps.
Here $E\times_{B}E$ and
$\mathbb{R}\times E$ are considered as smooth subcartesian spaces with respect to the product topologies.
A morphism from $(E,\pi,B)$ to
$(E^{\prime},\pi^{\prime},B^{\prime})$
consists of a pair of smooth maps
$\tilde{f}:E\rightarrow E^{\prime}$
and
$f:B\rightarrow B^{\prime}$
such that $\tilde{f}$ is $\mathbb{R}$-linear along fibers and the commutativity of the following diagram holds
$$
\xymatrix@=1.3cm{
  E \ar[d]_{\pi} \ar[r]^{\tilde{f}} & E^{\prime} \ar[d]^{\pi^{\prime}} \\
  B \ar[r]^{f} & B^{\prime}.  }
$$

\begin{defn}\label{v-p-bundle}
A $\C^{\infty}$-family of $\mathbb{R}$-vector spaces $(E,\pi,B)$
is called a \emph{$\C^{\infty}$-vector pseudobundle} if $E$ has a $\C^{\infty}$-sub-atlas
$$
\mathfrak{A}_{E}=
\{\phi_{\lambda}:U_{\lambda}\rightarrow
\mathbb{R}^{n_{\lambda}}\}_{\lambda\in\Lambda}
$$
satisfying the conditions:
\begin{itemize}
  \item [(i)] $U_{\lambda}=\pi^{-1}(\pi(U_{\lambda}))$ holds for each $\lambda\in\Lambda$;
  \item [(ii)] for each chart
      $\phi:U\rightarrow
      \mathbb{R}^{n}$ of $\mathfrak{A}_{E}$ the map $\phi$ determines a smooth map
      $$
      \underline{\phi}:\pi(U)
      \longrightarrow\mathbb{R}^{m}\,\,\,
      (m\leq n)
      $$
      such that
      $(\phi,\underline{\phi}):(U,\pi,\pi(U))
      \longrightarrow
      (\mathbb{R}^{n},\pi_{n,m},\mathbb{R}^{m})$
      is a morphism of $\C^{\infty}$-family of $\mathbb{R}$-vector spaces, where $\pi_{n,m}$ is the canonical projection from $\mathbb{R}^{n}$ to $\mathbb{R}^{m}$ via the first $m$ coordinates.
\end{itemize}
\end{defn}

Similar to smooth vector bundles, we can define smooth sections of a $\C^{\infty}$-vector pseudobundle (or a $\C^{\infty}$-family of $\mathbb{R}$-vector spaces) $\xi=(E,\pi,B)$ to be smooth maps $s:B\rightarrow E$ with $\pi\circ s=\mathrm{id}_{B}$.

In a more general setting, we can define a $\C^{0}$-family of $\mathbb{R}$-vector spaces over a smooth subcartesian spaces $(B,\mathfrak{A}_{B})$ as follows.
\begin{defn}\label{c0-family}
A \emph{$\C^{0}$-family of $\mathbb{R}$-vector spaces} over $B$ is a topological space $E$ together with a stratified continuous map $\pi$ of $E$ onto $B$ satisfying:
\begin{itemize}
  \item [(i)] for any $b\in B$ the fiber $\pi^{-1}(b)$ is a $\mathbb{R}$-vector space;
  \item [(ii)] the vector operators
    $$
    +:E\times_{B}E
    \longrightarrow E\,\,\,\,
    \mathrm{and}\,\,\,\,
    \cdot:\mathbb{R}\times E\longrightarrow E
    $$
    are continuous.
\end{itemize}
\end{defn}


\end{document}